\documentclass[10pt,reqno]{amsart}
\usepackage[utf8]{inputenc}
\usepackage{fullpage, tikz-cd} % add this since it gets rid of the super large margins
\usepackage{graphicx}
\usepackage{enumitem} % this is to make list look nice % http://ctan.org/pkg/enumitem
\usepackage{amsmath,amssymb,amsthm, amsfonts,verbatim, pdfpages}
\usepackage[colorlinks=true,hyperindex, linkcolor=blue, pagebackref=false, citecolor=cyan]{hyperref}
\usepackage{tikz}
\usepackage{cleveref}

\newcommand{\Z}{\mathbb{Z}}

\newcommand{\N}{\mathbb{N}}
\DeclareMathOperator{\Cone}{Cone}
\DeclareMathOperator{\Line}{line}
\DeclareMathOperator{\coker}{coker}
\DeclareMathOperator{\im}{im}

\newcommand{\xx}{\mathbf{x}}

\DeclareMathOperator{\intact}{ia}
\DeclareMathOperator{\extact}{ea}
\DeclareMathOperator{\IntAct}{IntAct}
\DeclareMathOperator{\ExtAct}{ExtAct}
\DeclareMathOperator{\wt}{wt}

\newcommand{\redden}[1]{\underline{\textcolor{red}{#1}}}
\newcommand{\CT}[2]{\mathrm{CT}_{#1,#2}}

\newcommand{\Coneplus}[1]{\mathrm{Cone}^{(+#1)}}

\newtheorem{theorem}{Theorem}[section]
\newtheorem{lemma}[theorem]{Lemma}
\newtheorem{definition}[theorem]{Definition}
\newtheorem{proposition}[theorem]{Proposition}
\newtheorem{corollary}[theorem]{Corollary}
\newtheorem{conjecture}[theorem]{Conjecture}
\newtheorem{question}[theorem]{Question}

\newtheorem{remark}[theorem]{Remark}

\begin{document}

\title{Sandpile groups for cones over trees}
\author{Victor Reiner}
\author{Dorian Smith}
\email{reiner@umn.edu, smi01055@umn.edu}
\address{School of Mathematics,
University of Minnesota -- Twin Cities}
\date{}

%\begin{document}
\thanks{Authors partially supported by NSF grant DMS-2053288.}
\subjclass{05C50, 05C25}
\keywords{tree, cone, sandpile, Laplacian, star, path, Fibonacci}

\maketitle

\begin{abstract}
Sandpile groups are a subtle graph isomorphism invariant, in the form of a finite abelian group,
whose cardinality is the number of spanning trees
in the graph.  We study their group structure
for graphs obtained by attaching a cone vertex to a tree.  For example, it is shown that 
the number of generators of the sandpile group is at
most one less than the number of leaves in the tree.
For trees on a fixed number of vertices,
the paths and stars are shown to
provide extreme behavior, not only for the
number of generators, but also for the number of
spanning trees, and for Tutte polynomial
evaluations that count the recurrent sandpile configurations by their numbers of chips. 
\end{abstract}

%%%%%%%%%%%%%%%%%%%%%%%%%%%%
%\tableofcontents
%%%%%%%%%%%%%%%%%%%%%%%%%%%%
\section{Introduction}

This paper studies sandpile groups for graphs which are cones over trees.  We first explain some background on
sandpile groups, some motivation 
for examining cones over trees, and then our
main results.

\subsection{Laplacians and sandpile groups}

Let $G$ be a finite, undirected, connected graph $G=(V,E)$ on vertices $V$ and edges $E$, with parallel edges allowed, but no self-loops.
Its {\it sandpile group} $K(G)$ is an interesting isomorphism invariant, a finite abelian group whose cardinality is the number $\tau(G)$ of spanning trees in $G$.  It can be defined starting with
the {\it Laplacian matrix} in $\Z^{V \times V}$, whose rows and columns are both indexed by
the vertices $v$ in $V$,  with $(v,v')$-entry given by 
\begin{equation}
\label{Laplacian-definition}
(L_G)_{v,v'}:=
\begin{cases}
\deg_G(v) & \text{ if }v=v',\\
-\# \{ \text{edges from }v\text{ to }v' \}& \text{ if }v \neq v'.
\end{cases}
\end{equation}
This Laplacian matrix may also be viewed as $L_G=\partial \circ \partial^T$ where $\partial: \Z^E \longrightarrow Z^V$ is the cellular boundary map for any orientation of the edges of $G$.
From here, one picks any choice of root vertex $v_0$ in $V$,
to define a reduced Laplacian matrix $\overline{L}_G$ in $\Z^{(V - \{v_0\}) \times (V - \{v_0\})}$,
and then the {\it sandpile group} is its integer cokernel:
\begin{equation}
    \label{sandpile-group-presentation}
K(G):=\coker (\overline{L}_G :\Z^{V  - \{v_0\}} \rightarrow \Z^{V  - \{v_0\}})=
\Z^{V  - \{v_0\}}/\im \overline{L}_G.
\end{equation}

The name {\it sandpile group} comes from the dynamics of the
{\it abelian sandpile model} from physics, which leads to certain distinguished representatives $\xx=(x_v)_{v \in V-\{v_0\}}$ in $\Z^{V  - \{v_0\}}$ for the cosets $\xx+\im\overline{L}_G$, called {\it recurrent chip configurations}; see Corry and Perkinson \cite[Chap. 6]{CorryPerkinson} and Klivans \cite[Chap. 4]{Klivans} for more
background.  Letting $\N:=\{0,1,2,\ldots\}$, one thinks of vectors $\xx$  in $\N^{V  - \{v_0\}}$ as {\it valid chip configurations}, with $x_v$ denoting the number chips piled at vertex $v$;  one does not keep track of the
number of chips at the root vertex $v_0$, which one can think of as $+\infty$ .
Say that vertex $v$ is {\it ready to fire} in a valid chip configuration $\xx$ if $x_v \geq \deg_G(v)$.  Then {\it firing} $\xx$ at $v$
subtracts from $\xx$ the $v^{th}$ column of $\overline{L}_G$; equivalently, one thinks of vertex $v$ as passing one chip along each of its incident edges to its neighbors in $G$.  Say $\xx$ is a {\it stable} chip configuration if $0 \leq x_v \leq \deg_G(v)-1$ 
for all $v$ in $V - \{v_0\}$, so that no vertex in $V - \{v_0\}$ is ready to fire.  One calls
$\xx$ a {\it recurrent} (or {\it critical}) chip configuration if it is both stable and has the property that
after firing the root vertex $v_0$ (passing one chip along each edge incident to $v_0$ to its neighbors), there exists some sequence of firings at vertices in $V - \{v_0\}$ that leads back to $\xx$.   These recurrent configurations then give a system of coset representatives for  $K(G)$, that is,
\begin{equation}
\label{criticals-index-cosets}
K(G)=\{ \xx+\im \overline{L}_G: \text{ recurrent }\xx \in \N^{V-\{v_0\}}\}
\end{equation}
allowing one way to compute in the group.

Recurrent chip configurations also relate to an interesting isomorphism invariant of $G$, its {\it Tutte polynomial} $T_G(x,y)$, reviewed in Section~\ref{Tutte-polynomial-section} below. Recall that $T_G(x,y)$ is a bivariate polynomial in $\Z[x,y]$ with nonnegative coefficients, which can
be defined either via a deletion-contraction recurrence, or as a bivariate generating function for the spanning trees $T \subset E$ according to two statistics.
Thus $T_G(1,1)=|K(G)|=\tau(G)$,
the spanning trees number.
More generally, one has the following result of C. Merino \cite{Merino}, conjectured by N. Biggs, showing $T_G(1,y)$ is (up to a power of $y$)
the generating function $\sum_{\xx} y^{\wt(\xx)}$ for recurrent configurations $\xx$ counted by
their number of chips or {\it weight}
$\wt(\xx):=\sum_{v \in V -\{v_0\}} x_v$;
see \cite[\S 3.3]{Klivans}, \cite[\S 14.4]{CorryPerkinson}.

\vskip.1in
\noindent
{\bf Theorem.} \cite[Thm. 3.6]{Merino}
{\it
For a graph $G$ with $m$ edges, and any root vertex $v_0$, one has
\begin{equation}
\label{Merino-theorm}
T_G(1,y) = 
y^{\deg_G(v_0)-m} \cdot
\sum_{\xx} y^{\wt(\xx)}.
\end{equation}
where in the sum here, $\xx$ runs over all recurrent chip configurations with respect to the root $v_0$.
}
\vskip.2in

\subsection{Motivation for cones over trees}

This paper was motivated by the following slightly vague question.
\vskip.1in
\noindent
{\bf Main Question.}
{\it 
How does the``shape" of the graph $G$ relate to the structure of its
sandpile group $K(G)$, including
the cardinality $\tau(G)=|K(G)|$, the minimal number of generators $\mu(G)$ for $K(G)$, 
and the Tutte evaluation $T_G(1,y)$ counting the recurrent configurations
from \eqref{criticals-index-cosets} by their number of chips?
}
\vskip.1in

The large literature on sandpile groups mostly leaves this unaddressed, focusing more on these aspects:
\begin{itemize}
\item Computing $K(G)$ explicitly for various special families of graphs, e.g.,  complete graphs \cite{Biggs} and their Cartesian products \cite{Bai, JNR}, complete bipartite \cite{Lorenzini} and multipartite graphs \cite{JNR}, circulant graphs, \cite{MednykhMednykh}, wheels \cite{Biggs}, threshold graphs \cite{ChristiansonR}, polygon chains and rings \cite{ChenMohar}, rook graphs \cite{DuceyGerhardWatson} Paley graphs \cite{ChandlerSinXiang}, cyclotomic strongly regular graphs \cite{Pantangi}, abelian group Cayley graphs \cite{DuceyJalil, GMMY},
conference graphs \cite{Lorenzini2}, and Erd\H{o}-R\'enyi random graphs $G(n,p)$ for fixed $p$ as $n$ approaches infinity \cite{Wood},
\item Understanding behavior of $K(G)$ under graph constructions, e.g., planar duality \cite{CoriRossin}, edge-subdivision/edge-duplication \cite{Lorenzini}, taking line graphs of regular graphs \cite{BMMPR}, forming covering spaces \cite{ReinerTseng}, taking quotients by dihedral groups \cite{GlassMerino}, adding a cone vertex to a Cartesian product of graphs \cite{AlfaroValencia}, adding cone vertices iteratively to graphs \cite{GoelPerkinson}.
\end{itemize}

Here are a few results that do address the above Main Question: 
\begin{itemize}
    \item[(a)] The presentation of $K(G)$ in \eqref{sandpile-group-presentation} shows that $\mu(G) \leq |V|-1$,
\item[(b)]  A dual presentation of $K(G)$ in terms of cycles \cite[\S 29]{Biggs-longpaper} 
shows that $\mu(G) \leq |E|-|V|+1$.  
\item[(c)] These two bounds on $\mu(G)$
are the {\it rank} and {\it corank} of the {\it graphic matroid} $M_G$
for $G$, and Wagner \cite{Wagner} proved that the group structure of $K(G)$ is a {\it matroid isomorphism invariant} of $M_G$.
\end{itemize}

Our goal is to start addressing the Main Question above
for the family of
graphs $G$ which are {\it cones over trees}:  $G=\Cone(T)$ is
obtained from a tree $T$ on vertex set $\{v_1,\ldots,v_n\}$ by adding a {\it cone vertex} $v_0$, with one edge $\{v_0,v_i\}$
for each $i=1,2,\ldots,n$.  Fixing $n$, all such graphs $G=\Cone(T)=(V,E)$ have 
\begin{align*}
    |V|-1&=n,\\ 
    |E|-|V|+1&=n-1,
\end{align*}
so that the above two upper bounds (a),(b) on $\mu(G)$, are fixed, and approximately equal.  This allows one to  focus on variation due to other parameters of the ``shape" of the tree $T$, such as its number of leaves, its degree sequence, etc.  We hope that trends in $K(G)$ for
$G=\Cone(T)$ may illuminate such trends for {\it all} graphs $G$; see
the discussion in Section~\ref{remarks-section}
below.

A bit more motivation comes from an observation of Alfaro and Valencia \cite{AlfaroValencia} that graphs of the form $\Cone(H)$ have
the cone vertex $v_0$ as an obvious distinguished candidate for the role
of the root vertex when doing chip-firing;  they examined $K(\Cone(H))$ in the case where $H=H_1 \times H_2$ is a Cartesian product graph.  Relatedly, Goel and Perkinson \cite{GoelPerkinson} studied how the sandpile group $K(G)$ changes when one takes iterated cones over $G$.  Lastly, Proudfoot and Ramos \cite{ProudfootRamos}
studied the topology of configuration spaces of
points in $\Cone(T)$ for trees $T$, showing
interesting behavior when one contracts edges in $T$.

\subsection{Main results}
Our first main result partly addresses the Main Question about $\mu(G)$ for cones on trees.  To state it, let $\{e_v\}$ be the standard basis vectors
for $\Z^{V - \{v_0\}}$, and denote by
$\{ \overline{e}_v \}$ their images within
the quotient presentation
of the sandpile group 
$K(G)=\Z^{V-\{v_0\}}/\im \overline{L}_G$
from \eqref{sandpile-group-presentation}.

\begin{theorem}
\label{generators}
For a tree $T$ on at least two vertices with $\ell(T)$ leaves, one has
$
\mu(\Cone(T)) \leq \ell(T)-1.
$

More precisely, choosing the cone vertex of $\Cone(T)$ as the root 
vertex $v_0$, the group $K(\Cone(T))$
is generated by 
$
\{\overline{e}_v: \text{ leaves }v \neq v_1\}
$
for any choice of a leaf vertex $v_1$ in $T$.
\end{theorem}

Our remaining results focus on bounds
derived from 
two extreme families of trees:
\begin{itemize}
    \item The
{\it path graph} $P_n$ with $n$ vertices.
Its cone graph $\Cone(P_n)$ is also called a {\it fan graph} \cite{bogdanowicz2008formulas, SELIG2023103663}.

\item The {\it star graph} $S_n=K_{1,n-1}$ with $n-1$ leaves, a complete bipartite graph.  The cone graph $\Cone(S_n)=K_{1,1,n-1}$, a complete tripartite graph, is also called a {\it thagomizer graph}  \cite{ProudfootRamos}.
\end{itemize}

Depicted below are the path $P_6$, the star $S_9$, 
and also the $p=6,s=5$ example from the interpolating family of {\it coconut trees} $\CT{p}{s}$ that will play a unifying role (and cautionary tale) for some of our results.  The {\it coconut tree} $\CT{p}{s}$ for $p,s \geq 1$ 
is obtained from a path on $p$ vertices $\pi_1,\pi_2,\ldots,\pi_p$ by adding $s$ leaf vertices $\sigma_1,\sigma_2,\ldots,\sigma_s$ attached to $\pi_p$ as their unique neighbor.

%\begin{figure}[h] 
\begin{center}
\begin{minipage}{.2\textwidth}
 \hspace*{-0.6\linewidth}
\begin{tikzpicture}
  [scale=.30,auto=left,every node/.style={circle,fill=black!20}]
 \node (pp1) at (-13,12) {$\pi_1$};
  \node (pp2) at (-10,12) {$\pi_2$};
  \node (pp3) at (-7,12) {$\pi_3$};
  \node (pp4) at (-4,12) {$\pi_4$};
  \node (pp5) at (-1,12) {$\pi_5$};
  \node (pp6) at (2,12) {$\pi_6$};
  \foreach \from/\to in {pp6/pp5,pp5/pp4,pp4/pp3,pp3/pp2,pp2/pp1}
    \draw (\from) -- (\to);
  %%%%%%%
    \node (ss1) at (15,16) {$\sigma_1$};
  \node (ss2) at (18,16) {$\sigma_2$};
  \node (ss3) at (18,12) {$\sigma_3$}; 
  \node (ss4) at (18,8) {$\sigma_4$}; 
  \node (ss5) at (15,8) {$\sigma_5$};
  \node (ss6) at (12,8) {$\sigma_6$};
  \node (ss7) at (12,12) {$\sigma_7$};
  \node (ss8) at (12,16) {$\sigma_8$};
   \node (ss9) at (15,12) {$\sigma_9$};
\foreach \from/\to in {ss9/ss1,ss9/ss2,ss9/ss3,ss9/ss4,ss9/ss5,ss9/ss6,ss9/ss7,ss9/ss8}
    \draw (\from) -- (\to);
  %%%%%%
 \node (p1) at (-10,0) {$\pi_1$};
  \node (p2) at (-7,0) {$\pi_2$};
  \node (p3) at (-4,0) {$\pi_3$};
  \node (p4) at (-1,0) {$\pi_4$};
  \node (p5) at (2,0) {$\pi_5$};
  \node (p6) at (5,0) {$\pi_6$};
  \node (s1) at (5,4) {$\sigma_1$};
  \node (s2) at (8,4) {$\sigma_2$};
  \node (s3) at (8,0) {$\sigma_3$}; 
  \node (s4) at (8,-4) {$\sigma_4$}; 
  \node (s5) at (5,-4) {$\sigma_5$}; 
  \foreach \from/\to in {p6/p5,p5/p4,p4/p3,p3/p2,p2/p1,p6/s1,p6/s2,p6/s3,p6/s4,p6/s5}
    \draw (\from) -- (\to);
\end{tikzpicture}
\end{minipage}
\end{center}
%\label{paths-stars-coconut-figure}
%\caption{The path $P_6$, the star $S_9$, and the coconut tree $\CT{6}{5}$.}
%\end{figure}

The second main result shows paths and stars give bounds on the minimal number of generators $\mu(\Cone(T))$. 

\begin{theorem}
\label{generators-bounds-theorem}
For trees $T$ on $n \geq 3$ vertices, the minimal
number of generators for
$K(\Cone(T))$ satisfies
$$
\begin{array}{cccccl}
\mu(\Cone(S_n)) &\geq& \mu(\Cone(T))&\geq &\mu(\Cone(P_n))\\
\Vert& & & &\Vert\\
n-2& & & &1.
\end{array}
$$
\end{theorem}

Theorem~\ref{generators-bounds-theorem} follows almost immediately by combining
Theorem~\ref{generators} with the following result,
that determines the spanning tree number and the sandpile group structure for
cones over paths, stars (and coconut trees).
To state it, introduce the abbreviation $\Z_m:=\Z/m\Z$ for finite cyclic groups.  
Also  recall
the {\it Fibonacci numbers}
$\{F_n\}_{n=0,1,2,\ldots}$  defined by
$F_0=0,F_1=1$ and $F_n=F_{n-1}+F_{n-2}$ for $n \geq 2$.  Lastly, for integers $p,s \geq 1$, let
\begin{align*}
m_{p,s}:=(s-2)F_{2p-1}+2F_{2p+1}=sF_{2p-1}+2F_{2p}.
\end{align*}

\begin{theorem}
\label{coconut-theorem}
For $p,s \geq 1$, the cone over the coconut tree $\CT{p}{s}$ has the following:
\begin{itemize}
\item[(i)]
$
\tau(\Cone(\CT{p}{s}))= 2^{s-1} m_{p,s}
$
\item[(ii)]
$
K(\Cone(\CT{p}{s})) \cong
\begin{cases}
    \Z_2^{s-1} \oplus \Z_{m_{p,s}} 
    & \text{ if }s=1\text{ or }p \equiv 2 \bmod{3},\\
    \Z_2^{s-2} \oplus \Z_{2m_{p,s}} 
    & \text{ if }s \geq 2\text{ and }p \equiv 0,1 \bmod{3}.
\end{cases}
$
\end{itemize}
\end{theorem}
\noindent
In particular, taking $(p,s)=(n-1,1)$ or $(p,s)=(1,n-1)$, respectively, gives
the spanning tree counts and sandpile
group structures for fans $\Cone(P_n)$ and thagomizer graphs
$\Cone(S_n)$:
\begin{align}
\label{fan-spanning-tree-number}
 \tau(\Cone(P_n))&=F_{2n} \,\, \text{ for }n \geq 1,\\
\label{thagomizer-spanning-tree-number}  
\tau(\Cone(S_n)) &= 2^{n-2} \cdot (n+1)
\,\, \text{ for }n \geq 2,\\
\notag& \\
\label{fan-sandpile-group}
 K(\Cone(P_n))&\cong\Z_{F_{2n}} \,\, \text{ for }n \geq 1,\\
\label{thagomizer-sandpile-group}  
K(\Cone(S_n)) &\cong \Z_2^{n-3} \oplus \Z_{2(n+1)}
\,\, \text{ for }n \geq 3.
\end{align}

\noindent
In fact, the spanning tree counts
\eqref{fan-spanning-tree-number}, \eqref{thagomizer-spanning-tree-number} were known previously; see
\cite{bogdanowicz2008formulas, SELIG2023103663}, and  \cite[eqn. (1)]{JNR}.

Our last main result gives coefficientwise
bounds on the 
Tutte polynomial evaluations $T_{\Cone(T)}(1,y)$, again coming from paths and stars.  However, note that the inequalities are {\it reversed} (!) compared to Theorem~\ref{generators-bounds-theorem}.
 
\begin{theorem}
    \label{tutte-bounds-theorem}
    For a tree $T$ on $n \geq 2$ vertices, 
    the weight enumerator of recurrent chip configurations
    $T_{\Cone(T)}(1,y)$ from \eqref{Merino-theorm} satisfies
the following coefficientwise inequalities as
polynomials in $\Z[y]$:
\begin{equation}
   \label{coefficentwise-Tutte-comparison}
   T_{\Cone(S_n)}(1,y) \,\,  \leq \,\, T_{\Cone(T)}(1,y) \,\, \leq \,\, T_{\Cone(P_n)}(1,y).
\end{equation}
In particular, setting $y=1$, this implies
\begin{equation}
\label{Urschel-result}
\begin{array}{cccccl}
\tau(\Cone(S_n)) &\leq& \tau (\Cone(T))&\leq &\tau(\Cone(P_n))\\
\Vert& & & &\Vert\\
2^{n-2}(n+1)& & & &F_{2n}
% & & & & (\approx \frac{1}{\sqrt{5}} \cdot (\varphi^2)^n)&
\end{array}
\end{equation}
\end{theorem}
\noindent
The authors thank J. Urschel \cite{Urschel} for proofs of both numerical inequalities in \eqref{Urschel-result} when they were still conjectural.  These proof ideas generalized, with some work, to prove 
the coefficientwise inequalities \eqref{coefficentwise-Tutte-comparison}.

One can say more about the polynomials
appearing as the lower and upper bounds in \eqref{coefficentwise-Tutte-comparison}.
For example, Corollary~\ref{fan-doubled-fan-corollary} will show that
$
T_{\Cone(P_n)}(1,y)=F_{2n}(y),
$
where $F_n(y)$ is defined via a Fibonacci-like recursion:
\begin{align}
\notag
    F_0(y)&:=0, \,\, F_1(y):=1, \\
\label{Fibonacci-polynomial-definition}
F_n(y) &:= 
\begin{cases}
F_{n-1}(y) + F_{n-2}(y) &\text{ for }n\text{ even,}\\
F_{n-1}(y) + yF_{n-2}(y) &\text{ for }n\text{ odd.}
\end{cases}
\end{align}
Interestingly,
Selig \cite[Thm. 5.12]{SELIG2023103663}
gives a combinatorial interpretation for
the expansion coefficients of $F_{2n}(y)=T_{\Cone(P_n)}(1,y)$, but his methods are purely bijectve, and the above Fibonacci-like recurrence for $F_n(y)$ seems to play no role there. 
Meanwhile, we will show in Lemma~\ref{star-inequality-lemma} below that
$
T_{\Cone(S_n)}(1,y)=S_n(y),
$
where $S_n(y)$ is defined via this recursion:
\begin{align}
\notag
S_1(y)&:=1,\\
\label{star-polynomial-definition}
S_n(y)&:=(y+1) S_{n-1}(y) + 2^{n-2} \text{ for }n \geq 2.
\end{align}

 The structure of the paper is as follows.
 Section~\ref{generators-section} proves Theorem~\ref{generators}.  Section~\ref{Tutte-polynomial-section} reviews Tutte polynomials 
 and then proves a few deletion-contraction lemmas relevant for cones on trees.
 Section~\ref{coconut-part1-section}
 uses these lemmas to prove the spanning tree counts for coconut trees asserted in Theorem~\ref{coconut-theorem}(i), while Section~\ref{coconut-part2-section} employs Theorem~\ref{generators}
 to compute their sandpile groups asserted in
 Theorem~\ref{coconut-theorem}(ii).
 Section~\ref{tutte-bounds-section} proves
 Theorem~\ref{tutte-bounds-theorem}.
 Section~\ref{remarks-section} contains
 further discussion and questions.

%%%%%%%%%%%%%%%%%%%%%%%%
\section{Proof of Theorem~\ref{generators}}
\label{generators-section}

Recall the statement of the theorem,
involving the images 
$\{ \overline{e}_v \}$ of the  standard basis vectors
for $\Z^{V - \{v_0\}}$ within the
sandpile group presented via
$K(G)=\Z^{V-\{v_0\}}/\im \overline{L}_G$
as in \eqref{sandpile-group-presentation}.

\vskip.1in
\noindent
{\bf Theorem~\ref{generators}}
{\it 
For a tree $T$ on at least two vertices with $\ell(T)$ leaves, one has
$$
\mu(\Cone(T)) \leq \ell(T)-1.
$$

More precisely, choosing the cone vertex $v_0$ in $\Cone(T)$ as the root 
vertex, the group $K(\Cone(T))$
is generated by 
$
\{\overline{e}_v: \text{ leaves }v \neq v_1\}
$
for any choice of a leaf vertex $v_1$ in $T$.}
\vskip.1in

\begin{proof}
Use the choice of the leaf vertex $v_1$ to define a partial order $<$ on the vertex set $V - \{v_0\}$ of $T$, decreeing $v \leq v'$ if $v'$ 
lies on the unique path from the leaf $v$ to the root $v_1$ in $T$.  The
partial order has
$v_1$ as its unique $<$-maximal element, and the $<$-minimal elements are the remaining $\ell(T)-1$ leaves $v \neq v_1$ of $T$.  

We claim that it suffices to show, for each {\it $<$-nonminimal} vertex $v$ of $T$
that  $\overline{e}_v$ can be
expressed in $K(G)$ as lying in the $\Z$-span of $\{ \overline{e}_{v'}: v' \lneqq v$\}.  This would then show, by induction on any linear extension of $<$, that every vertex $u$ in $V -\{v_0\}$ has
$\overline{e}_u$ lying in the $\Z$-span of 
$\{ 
\overline{e}_{v} : <\text{-minimal }v
\}
=\{ \overline{e}_{v} : \text{leaves }v \neq v_1 \}$.

Assume that $v$ is a $<$-non-minimal vertex of $T$, so that $v$ has at least one neighboring vertex $v'$ in $T$ with $v' < v$.  In other words, $v'$ is a child of $v$ when regarding $v_1$ as the root of the tree $T$.  Note that all other neighbors $v'' \neq v$ of $v'$ in $T$ are children of $v'$
and satisfy $v'' < v$, since both $v',v$  lie on their unique path from $v''$ to the root $v_1$ in $T$.  See this example:

\begin{center}
\begin{minipage}{.2\textwidth}
 \hspace*{-0.6\linewidth}
\begin{tikzpicture}
  [scale=.28,auto=left,every node/.style={circle,fill=black!20}]

  \node (v1) at (15,16) {$v_1$};
  \node (v) at (18,8) {$v$}; 
  \node (v') at (15,4) {$v'$}; 
  \node (v1'') at (11,0) {$v_1''$}; 
  \node (v2'') at (15,0) {$v_2''$}; 
  \node (v3'') at (19,0) {$v_3''$}; 
  \node (d) at (21,4) {$ $}; 
  \node (e) at (11,-4) {$ $}; 
  \node (f) at (19,-4) {$ $}; 
  \node (a) at (12,8) {$ $};
  \node (c) at (12,4) {$ $};
  \node (b) at (15,12) {$ $};
\foreach \from/\to in {v1/b,b/a,b/v,v/v',a/c,v'/v1'',v'/v2'',v'/v3'',v/d,v1''/e,v3''/f}
    \draw (\from) -- (\to);
\end{tikzpicture}
\end{minipage}
\end{center}

Then $\bar{e}_v$ is equal in $K(G)$ to the result of adding the $(v')^{th}$ column of $\overline{L}_{\Cone(T)}$ to itself:
$$
%\begin{aligned}
\bar{e}_v 
\equiv \bar{e}_v + 
\left( 
\deg_{\Cone(T)}(v') \cdot \bar{e}_{v'}
- \sum_{ \substack{v'' \in V-\{v_0\}:\\ \{v',v''\} \in T}} \bar{e}_{v''}
\right)
\equiv
\deg_{\Cone(T)}(v') \cdot \bar{e}_{v'}
- \sum_{ \substack{v'' \in V-\{v_0, v'\}:\\ \{v',v''\} \in T}} \bar{e}_{v''}
%\end{aligned}
$$
This achieves the goal, since all of the $v',v''$ appearing on the far right have $v',v''<v$.
\end{proof}

The bound in Theorem~\ref{generators}
raises the following question.

\begin{question} 
\label{mu-bound-question}
Which trees $T$ 
    achieve equality in the bound of Theorem~\ref{generators}, that is,
    $\mu(\Cone(T))=\ell(T)-1$?
\end{question}

\noindent
We suspect the answer is subtle.  Remark~\ref{coconut-tree-mu-remark} below deduces from Theorem~\ref{coconut-theorem} that
{\it coconut trees} $T=\CT{p}{s}$ always have $\mu(\Cone(T))=\ell(T)-1$ or $\ell(T)-2$,
but the distinction between the two cases is already slightly tricky. 
Other examples show that $\mu(\Cone(T))$ need not always
lie in $\{\ell(T)-1,\ell(T)-2\}$.
For example, the tree $T$
below has $\ell(T)=4$ but $K(\Cone(T)) = \Z_{332}$,
so $\mu(\Cone(T))=1 = \ell(T)-3$:

\begin{center}
\begin{minipage}{.2\textwidth}
 \hspace*{-0.6\linewidth}
\begin{tikzpicture}
  [scale=.20,auto=left,every node/.style={circle,fill=black!20}]
  %%%%%%
  \node (v1) at (-10,0) {};
  \node (v2) at (-15,3) {};
  \node (v3) at (-15,-3) {};
  \node (v4) at (-5,0) {}; 
  \node (v5) at (0,3) {};
   \node (v6) at (0,-3) {};
    \node (v7) at (5,-3) {};
  \foreach \from/\to in {v1/v4,v1/v2,v1/v3,v4/v5,v4/v6,v6/v7}
    \draw (\from) -- (\to);
\end{tikzpicture}
\end{minipage}
\end{center}
Small examples of complete binary 
trees $T$ suggest
$\ell(T) - \mu(\Cone(T))$ can grow arbitrarily large.
We know of no lower bound of the form
$\mu(\Cone(T)) \geq f(\ell(T))$ for some increasing and unbounded function $f(x)$.  
\begin{question}
    Are there trees
$\{T_n\}_{n=1,2,\ldots}$ all having $\mu(\Cone(T_n))=1$
where $\ell(T_n)$ grows without bound?
\end{question}

%%%%%%%%%%%%%%%%%%%%%%%%%%%%%%%%%%%%%%%%%%%%%
\section{The Tutte polynomial and deletion-contraction lemmas}
\label{Tutte-polynomial-section}

We recall two definitions of the
Tutte polynomial $T_G(x,y)$, and derive recurrences for cones over trees.

\subsection{Two definitions of the Tutte polynomial}
\label{tutte-review-section}

The first definition of $T_G(x,y)$ is recursive.

\begin{definition} \rm
\label{recursive-definition-of-Tutte-polynomia}
For a graph $G=(V,E)$ an edge $e$ is
a {\it loop} if its endvertices are equal,
and a {\it coloop} (or {\it isthmus} or {\it bridge}) if its deletion $G \setminus e$
has one more connected component
than $G$.

One can then define the {\it Tutte polynomial} $T_G(x,y) \in \mathbb{Z}[x,y]$ recursively on $|E|$, via these rules:

\begin{equation}
    \label{Tutte-recursive-definition}
T_G(x,y) = 
\begin{cases}
T_{G\setminus e}(x,y)+ T_{G/e}(x,y)
&\text{if  }e \text{ is neither a loop nor a coloop}, \\
y \cdot T_{G \setminus e}(x,y)
&\text{if  e is a loop.}\\
x \cdot T_{G/e} (x,y)
&\text{if  e is a bridge.}\\
1&\text{if }G\text{ has no edges, that is, }E=\varnothing.
\end{cases}
\end{equation}
\end{definition}
It can be shown that $T_G(x,y)$ has the following
property:  when $v$ is a cut-vertex in $G$, meaning that
its deletion $G-v$ has several components with vertex sets $C_1,C_2,\ldots,C_m$, then
\begin{equation}
\label{cut-vertex-Tutte-fact}
T_G(x,y) = \prod_{i=1}^m T_{G_i}(x,y)
\end{equation}
where $G_i$ denotes the vertex-induced subgraph of $G$ on vertex set $C_i \cup \{v\}$.

The definition above does not make it clear that $T_G(x,y)$ is well-defined, independent of deletions and contractions used to compute it recursively.  This can be 
remedied by Tutte's original definition from \cite{Tutte}, recalled here, as a bivariate generating function for spanning trees $T \subset E$,
counted by two statistics; see, e.g., Klivans \cite[\S 3.2.3]{Klivans}.

\begin{definition} \rm
\label{Tuttes-definition}
Fix any linear ordering $\prec$ of the edges $E$ of $G$.  Then one can define for any spanning tree $T \subset E$ its subset of {\it $\prec$-internally active} edges, and the set of {\it $\prec$-externally active} edges in $E - T$  as follows:
\begin{align}
\IntAct_\prec(T)
&:=\{e \in T: 
e \prec f \text{ whenever }(T  - e) \cup f\text{ is a spanning tree}\},
\\
\ExtAct_\prec(T)&:=\{f \in  
E - T: 
f \prec e \text{ whenever }(T  - e) \cup f\text{ is a spanning tree}\}.
\end{align}
The cardinalities of these sets are called the {\it internal activity} and {\it external activity} of $T$ \begin{align}
    \intact_\prec(T)&:=|\IntAct_\prec(T)|,\\
    \extact_\prec(T)&:=|\ExtAct_\prec(T)|,
\end{align}
and one can define $T_G(x,y)$ as this bivariate generating function, summing over spanning trees $T$ of $G$:
\begin{equation}
    \label{Tutte-activity-definition}
T_G(x,y) = \sum_T x^{\intact_\prec(T)} y^{\extact_\prec(T)}.
\end{equation}
\end{definition}
Setting $x=1$ in \eqref{Tutte-activity-definition} gives a consequence for later use (in Section~\ref{tutte-bounds-section}, proof of Lemma~\ref{star-inequality-lemma}):
\begin{equation}
    \label{ext-act-formula}
T_G(1,y)=\sum_T y^{\extact_\prec(T)}.
\end{equation}

\begin{remark} \rm
We remark that Definition~\ref{Tuttes-definition} of
$T_G(x,y)$ looks like it depends upon the ordering $\prec$, but Tutte showed that it does not.  There is also a third well-known definition of $T_G(x,y)$ that avoids both issues present in Definitions~\ref{recursive-definition-of-Tutte-polynomia}, ~\ref{Tuttes-definition}, namely Crapo's {\it corank-nullity formula} for $T_G(x,y)$ \cite{Crapo}.
\end{remark}

\subsection{Some deletion-contraction lemmas for cones of trees}

Many proofs involving trees $T$ use a leaf-induction strategy, in which picks a leaf vertex $\ell$ of $T$, and then considers the
subtree $T-\ell$ on one fewer vertex, obtained by deleting $\ell$ along with its unique incident edge.  
In computing Tutte polynomials of cones over trees with cone vertex $v_0$, the leaf-induction strategy is useful for performing deletion and contraction of $\Cone(T)$ along the edge $e=\{v_0,\ell\}$, leading one to consider the following auxiliary family of graphs.

\begin{definition} \rm
Given a vertex $v$ in a tree $T$, define the graph $\Coneplus{v}(T)$ to be the graph obtained from $\Cone(T)$ by creating a second parallel copy of the edge $\{v_0,v\}$.
\end{definition}

\begin{proposition}
For $T$ a tree, $\ell$ a leaf vertex of $T$, and $v$ the unique neighbor vertex of $\ell$ in $T$, one has
\begin{align}
\label{leaf-del-con}
T_{\Cone(T)}(x,y)
&=x \cdot T_{\Cone(T -\ell)}(x,y)
 + T_{\Coneplus{v}(T-\ell)}(x,y),\\
 \label{doubled-leaf-del-con}
 T_{\Coneplus{\ell}(T)}(x,y)
&=T_{\Cone(T)}(x,y)
 + y \cdot T_{\Coneplus{v}(T-\ell)}(x,y).
\end{align}
Therefore upon setting $x=1$, one has
\begin{align} 
\label{specialized-leaf-del-con}
 T_{\Cone(T)}(1,y)
&=T_{\Cone(T -\ell)}(1,y)
 + T_{\Coneplus{v}(T-\ell)}(1,y),\\
\label{specialized-doubled-leaf-del-con}
T_{\Coneplus{\ell}(T)}(1,y)
&=T_{\Cone(T)}(1,y)
 + y \cdot T_{\Coneplus{v}(T-\ell)}(1,y).
\end{align}
\end{proposition}
\begin{proof}
It suffices to prove \eqref{leaf-del-con}, \eqref{doubled-leaf-del-con}.
To prove \eqref{leaf-del-con},
we delete and contract on the
edge $e=\{v_0,\ell\}$ in the graph $\Cone(T)$.
The deletion $\Cone(T)\setminus e$ is isomorphic to  $\Cone(T-\ell)$ with an extra
leaf $\ell$ attached to the unique neighbor $v$ in $T$.  This makes the edge $\{\ell,v\}$ a coloop in $\Cone(T)\setminus e$, and hence by \eqref{Tutte-recursive-definition}, one has
$$
T_{\Cone(T)\setminus e}(x,y)=x \cdot T_{\Cone(T-\ell)}(x,y).
$$
On the other hand, the contraction $\Cone(T)/ e$ is isomorphic to  $\Coneplus{v}(T-\ell)$.  Hence \eqref{Tutte-recursive-definition} implies \eqref{leaf-del-con}.

Similarly, to prove \eqref{doubled-leaf-del-con},
note that if one deletes and contracts on one of the two parallel copies of the edge $e=\{v_0,\ell\}$ in the graph $\Coneplus{\ell}(T)$,
the deletion $\Coneplus{\ell}(T)\setminus e$ is isomorphic to  $\Cone(T)$.  The contraction $\Coneplus{\ell}(T)/e$ is isomorphic to $\Coneplus{v}(T-\ell)$ with a loop attached to the cone vertex, so \eqref{Tutte-recursive-definition} says
$$
T_{\Coneplus{\ell}(T)/e}(x,y) = y \cdot T_{\Coneplus{v}(T-\ell)}(x,y).
$$
Hence \eqref{Tutte-recursive-definition} similarly implies \eqref{doubled-leaf-del-con}.
\end{proof}

As a corollary, we deduce an assertion from the Introduction.  Recall that
\eqref{Fibonacci-polynomial-definition} 
defined a sequence of polynomials
$F_n(y)$ in $\Z[y]$ via a generalization
of the Fibonacci recurrence:
$$
F_0(y):=0, F_1(y):=1, \quad
F_n(y) := 
\begin{cases}
F_{n-1}(y) + F_{n-2}(y) &\text{ for }n\text{ even,}\\
F_{n-1}(y) + yF_{n-2}(y) &\text{ for }n\text{ odd.}
\end{cases}
$$

\begin{corollary}
\label{fan-doubled-fan-corollary}
Consider the path $P_n$ for $n \geq 1$,
and $\ell$ an end vertex of $P_n$. Then
\begin{align*}
T_{\Cone(P_n)}(1,y)&=F_{2n}(y),\\
T_{\Coneplus{\ell}(P_n)}(1,y)&=F_{2n+1}(y).
\end{align*}
\end{corollary}
\begin{proof}
Defining a sequence of polynomials
$\{\hat{F}_{n}(y)\}_{n=1,2,\ldots}$ via
\begin{align*}
\hat{F}_{2n}(y)&:=T_{\Cone(P_n)}(1,y), \\
\hat{F}_{2n+1}(y)&:=T_{\Coneplus{\ell}(P_n)}(1,y),
\end{align*}
we wish to show $\hat{F}_n(y)=F_n(y)$ for $n \geq 2$.
Since $T=P_n$ has $T-\ell=P_{n-1}$, one can
check that the recurrences \eqref{specialized-leaf-del-con} 
and \eqref{specialized-doubled-leaf-del-con} become
\begin{align*}
\hat{F}_n(y) &=
\hat{F}_{n-1}(y) + \hat{F}_{n-2}(y) \text{ for }n\text{ even,} \\
\hat{F}_n(y)&=\hat{F}_{n-1}(y) + y\hat{F}_{n-2}(y) \text{ for }n\text{ odd.}
\end{align*}
Hence $\{ \hat{F}_n(y)\}$
satisfies the recursion \eqref{Fibonacci-polynomial-definition} for $\{F_n(y)\}$.  After directly checking the bases cases
\begin{align*}
\hat{F}_2(y):=T_{\Cone(P_1)}(1,y)&=1=F_2(y),\\
\hat{F}_3(y):=T_{\Cone^{(+\ell)}(P_1)}(1,y)&=1+y=F_3(y),
\end{align*}
one concludes by induction on $n$ that $\hat{F}_n(y)=F_n(y)$ for $n \geq 2$.
\end{proof}

Here is a second corollary, proving another assertion from the Introduction.

\begin{corollary}
\label{star-double-star-corollary}
Consider the star $S_n$ for $n \geq 1$,
with $v$ the central vertex connected to all others.  Then
\begin{align}
\label{thagomizer-tree-count}
\tau(\Cone(S_n))&=2^{n-2}(n+1),\\
\label{doubled-thagomizer-tree-count}
\tau(\Coneplus{v}(S_n))&=2^{n-2}(n+3).
\end{align}
\end{corollary}
\begin{proof}
Prove  \eqref{thagomizer-tree-count},\eqref{doubled-thagomizer-tree-count}
by simultaneous induction on $n$. Both are easily checked in the base case $n=1$.  

In the inductive step where $n \geq 2$, we first prove \eqref{thagomizer-tree-count}.  Bearing in mind that $\tau(G)=T_G(1,1)$, one can set $x=y=1$ in \eqref{leaf-del-con},
and apply it setting $T=S_n$ and setting $\ell$ to be any of its leaves to obtain
\begin{align*}
\tau(\Cone(S_n))
&=\tau(\Cone(S_{n-1}))+\tau(\Coneplus{v}(S_{n-1}))\\
&\overset{(*)}{=}2^{n-3}(n)+2^{n-3}(n+2)\\
&=2^{n-3}(2n+2)=2^{n-2}(n+1),
\end{align*}
where equality (*) used both assertions 
\eqref{thagomizer-tree-count},\eqref{doubled-thagomizer-tree-count} inductively.
In the inductive step for proving \eqref{doubled-thagomizer-tree-count},
we delete and contract on one copy of the doubled
edge $e=\{v_0,v\}$ in $\Coneplus{v}(S_{n})$, 
to obtain 
\begin{align*}
\tau(\Coneplus{v}(S_{n}) )
&=\tau( \Coneplus{v}(S_{n}) \setminus e )
+\tau( \Coneplus{v}(S_{n})/e )\\
&\overset{(**)}{=}
2^{n-2}(n+1)+2^{n-1}\\
&=2^{n-2}(n+3)
\end{align*}
where equality (**) used two observations:
\begin{itemize}
 \item $\Coneplus{v}(S_{n}) \setminus e \cong \Cone(S_n)$,
    so that one can apply \eqref{thagomizer-tree-count} inductively, and
 \item $\Coneplus{v}(S_{n})/e$ is isomorphic to the multigraph obtained from the star tree $S_{n-1}$ by duplicating each of its $n-1$ edges with a parallel edge.  Hence it has $2^{n-1}$
 spanning trees.$\qedhere$
\end{itemize}
\end{proof}

%%%%%%%%%%%%%%%%%%%%%%%%
\section{Proof of Theorem~\ref{coconut-theorem}(i)}
\label{coconut-part1-section}

This section proves Theorem ~\ref{coconut-theorem}(i),
counting spanning trees for graphs which are cones over the coconut trees defined in the Introduction.  
It then also discusses an interesting property on the asymptotic growth of these spanning tree numbers.

Recall the statement of the theorem here, which used
for integers $p,s \geq 1$ the abbreviation
$$
m_{p,s}:=(s-2)F_{2p-1}+2F_{2p+1}=sF_{2p-1}+2F_{2p}.
$$

\vskip.1in
\noindent
{\bf Theorem~\ref{coconut-theorem}(i).}
{\it
For $p,s \geq 1$, the cone over the coconut tree $\CT{p}{s}$
has 
$
\tau(\Cone(\CT{p}{s}))= 2^{s-1} m_{p,s}.
$
}
\vskip.1in
\noindent
\begin{proof}
Use induction on $p$ to simultaneously prove the assertion of the theorem and a second formula:
\begin{align}
\label{CT-spanning-tree-number}
\tau (\Cone(\CT{p}{s})) &= 2^{s-1} \cdot (2{F}_{2p+1} + (s-2){F}_{2p-1}),\\
\label{DCT-spanning-tree-number}
\tau (\Coneplus{\pi_1}(\CT{p}{s})) &= 2^{s-1} \cdot (2{F}_{2p+2} + (s-2){F}_{2p}).
\end{align}
Recall $\pi_1$ is the leaf vertex of $\CT{p}{s}$ which is farthest from all of the star vertices $\sigma_1,\sigma_2,\ldots,\sigma_s$, except when $p=1$; in that case, $\CT{1}{s}\cong S_{s+1}$ and $\pi_1$ is generally not a leaf, but rather the central vertex of the star. 

\vskip .1in
\noindent
{\sf Base case:} $p=1$.\\
Here $\CT{1}{s} \cong S_{s+1}$, and one can check \eqref{CT-spanning-tree-number}, \eqref{DCT-spanning-tree-number} agree with \eqref{thagomizer-tree-count}, \eqref{doubled-thagomizer-tree-count}
from Corollary~\ref{star-double-star-corollary}.

\vskip .1in
\noindent
{\sf Inductive step:} $p \geq 2$.\\
We will apply the specialization to $x=y=1$
of \eqref{leaf-del-con},
\eqref{doubled-leaf-del-con},
which assert
\begin{align*}
\tau(\Cone(T))
&=\tau(\Cone(T -\ell))
 + \tau(\Coneplus{v}(T-\ell))\\
\tau(\Coneplus{\ell}(T))
&=\tau(\Cone(T))
 + \tau(\Coneplus{v}(T-\ell)).
\end{align*}

When applied to the leaf $\ell=\pi_1$ in the tree $T=\CT{p}{s}$, they give the following
\begin{align*}
\tau(\Cone(\CT{p}{s}))
&=\tau(\Cone(\CT{p-1}{s}))
 + \tau(\Coneplus{\pi_1}(\CT{p-1}{s}))\\
\tau(\Coneplus{\pi_1}(\CT{p}{s}))
&=\tau(\Cone(\CT{p}{s}))
 + \tau(\Coneplus{\pi_1}(\CT{p-1}{s})).
\end{align*}

Thus if one assumes \eqref{CT-spanning-tree-number},\eqref{DCT-spanning-tree-number} both hold for $p-1$, then one can deduce that \eqref{CT-spanning-tree-number} holds for $p$:
\begin{align*}
\tau(\Cone(\CT{p}{s})&=
\tau(\Cone(\CT{p-1}{s}))
 + \tau(\Coneplus{\pi_1}(\CT{p-1}{s}))\\
&=2^{s-1} \cdot (2{F}_{2p-1} + (s-2){F}_{2p-3})
+ 2^{s-1} \cdot (2{F}_{2p} + (s-2){F}_{2p-2})
\\
&=2^{s-1} \cdot (2({F}_{2p-1}+F_{2p}) + (s-2)({F}_{2p-3}+F_{2p-2}))\\
&=2^{s-1} \cdot (2{F}_{2p+1} + (s-2){F}_{2p-1}).
\end{align*}

And if one assumes \eqref{CT-spanning-tree-number} holds for $p$ and \eqref{DCT-spanning-tree-number} holds for $p-1$, then one deduces that \eqref{DCT-spanning-tree-number} also holds for $p$:
\begin{align*}
\tau(\Coneplus{\pi_1}(\CT{p}{s}))&=
\tau(\Cone(\CT{p}{s}))
 + \tau(\Coneplus{\pi_1}(\CT{p-1}{s}))\\
&= 2^{s-1} \cdot (2{F}_{2p+1} + (s-2){F}_{2p-1})
+2^{s-1} \cdot (2{F}_{2p} + (s-2){F}_{2p-2})\\
&=2^{s-1} \cdot (2({F}_{2p+1}+F_{2p}) + (s-2)({F}_{2p-1}+F_{2p-2}))\\
&=2^{s-1} \cdot (2{F}_{2p+2} + (s-2){F}_{2p}).
\qedhere
\end{align*}
\end{proof}

We pause here to discuss a consequence
of Theorem~\ref{coconut-theorem}(i)
on {\it asymptotic} growth of $\tau(\Cone(T))$.
Recall that the {\it Fibonacci numbers}
$\{F_n\}_{n=0,1,2,\ldots}$ have this
explicit formula
$$
F_n=\frac{\varphi^n - \psi^n}{\sqrt{5}}
\approx \frac{1}{\sqrt{5}} \cdot \varphi^n
$$
where $\varphi:=\frac{1+\sqrt{5}}{2}$
%$\approx 1.618\ldots$
is the {\it golden ratio}, and $\psi:=\frac{1-\sqrt{5}}{2}$.
Therefore when one considers the lower and upper bounds given in 
\eqref{Urschel-result}
on $\tau(\Cone(T))$ for trees $T$ with $n$ vertices
$$
2^{n-2}(n+1)=\tau(S_n)
\,\, \leq \,\,
\tau(\Cone(T))
\,\, \leq \,\,
\tau(P_n) = F_{2n} \approx \frac{(\varphi^2)^n} {\sqrt{5}},
$$
one finds both have approximately geometric growth, whose geometric ratios are these limits of the $n^{th}$ roots:
\begin{align*}
   \lim_{n \rightarrow \infty} \sqrt[n]{\tau(S_n)} &= \lim_{n \rightarrow \infty} \sqrt[n]{2^{n-2}(n+1)}=2,\\
   \lim_{n \rightarrow \infty} \sqrt[n]{\tau(P_n)} &= \lim_{n \rightarrow \infty} \sqrt[n]{F_{2n}}=\varphi^2 \approx 2.618....
\end{align*}
Consequently, in any family of trees $T_n$ with $n$ vertices, for $n$ very large one has approximate inequalities
$$
2 \lessapprox \sqrt[n]{\tau(\Cone(T_n))} \lessapprox \varphi^2.
$$
One might ask whether for any real value $\beta$ in the range $[2,\varphi^2]$ there exists a sequence $T_1,T_2,\ldots$ of trees  
with $\lim_{n \rightarrow \infty} \sqrt[n]{\tau(T_n)}=\beta.$  This is indeed the case for a subfamily of coconut trees defined as follows. Given $\beta$ in $[2,\varphi^2]$,
so that $\frac{\beta}{2}$ lies in $[1,\frac{\varphi^2}{2}]$,
define a real number $\alpha$ in $[0,1]$ via
$
\alpha:=\log_{\frac{\varphi^2}{2}}
\left( \frac{\beta}{2}\right),
$
so that 
\begin{equation}
    \label{alpha-beta-relation}
\frac{\beta}{2}=\left(\frac{\varphi^2}{2}\right)^\alpha.
\end{equation}
We will consider the family of coconut trees $T_n:=\CT{p_n}{s_n}$ with $p_n+s_n=n$ where 
$p_n \approx \alpha n$, say $p_n:=\lfloor \alpha n \rfloor$ to be definite.  Then
$p_n+s_n=n$ requires $s_n:=\lceil (1-\alpha) n \rceil \approx (1-\alpha) n$.

\begin{proposition}
For any real number $\beta$ in $[2,\varphi^2]$, with the above definition of the family $\{ T_n \}$, one has 
   $$
   \lim_{n\rightarrow \infty} \sqrt[n]{\tau(\Cone(T_n))}
   =\beta.
   $$
\end{proposition}

\begin{proof}
Theorem~\ref{coconut-theorem}(i) shows that for $p, s$ both large,
\begin{align*}
\tau(\Cone(\CT{p}{s}))
 &=2^{s-1} (sF_{2p-1} +2F_{2p})\\
 &\approx 2^{s-1} \frac{1}{\sqrt{5}} \left( s \cdot \varphi^{2p-1} +2\varphi^{2p}\right)\\
 &=2^s \cdot (\varphi^2)^p \cdot \frac{s+2\varphi}{2 \sqrt{5} \varphi}\\
\end{align*}
Consequently, when $n$ is large and one chooses $p_n \approx \alpha n$ and $s_n \approx (1-\alpha)n$ as above, one has
\begin{align*}
 \sqrt[n]{\tau(\Cone(\CT{p_n}{s_n})}
 & \approx 
 \sqrt[n]{
  2^{(1-\alpha)n}
 } \cdot
 \sqrt[n]{
 (\varphi^2)^{\alpha n}
 } \cdot
 \sqrt[n]{
 \frac{(1-\alpha)n +2\varphi}{2 \sqrt{5} \varphi} 
 } \\
& = 2^{1-\alpha} \cdot \left( \varphi^2 \right)^\alpha \cdot \sqrt[n]{
 \frac{(1-\alpha) n +2\varphi}{2 \sqrt{5} \varphi} 
 } 
\end{align*}
In the limit as $n \rightarrow \infty$, the
last factor approaches $1$, while \eqref{alpha-beta-relation} shows that the first two factors give $\beta$:
$$
2^{1-\alpha} \cdot \left( \varphi^2 \right)^\alpha = 
2 \left( \frac{\varphi^2}{2} \right)^\alpha
= \beta. \qedhere
$$
\end{proof}

%%%%%%%%%%%%%%%%%%%%%%%%
\section{Proof of Theorem~\ref{coconut-theorem}(ii)}
\label{coconut-part2-section}

This section proves Theorem ~\ref{coconut-theorem}(ii),
on the structure of the sandpile group for
graphs which are cones over coconut trees.
We recall here the statement, again involving
$
m_{p,s}:=(s-2)F_{2p-1}+2F_{2p+1}=sF_{2p-1}+2F_{2p}.
$
\vskip.1in
\noindent
{\bf Theorem~\ref{coconut-theorem}(ii).}
{\it 
For $p,s \geq 1$, the cone over the coconut tree $\CT{p}{s}$
has 
$$
K(\Cone(\CT{p}{s})) \cong
\begin{cases}
    \Z_2^{s-1} \oplus \Z_{m_{p,s}} 
    & \text{ if }s=1\text{ or }p \equiv 2 \bmod{3},\\
    \Z_2^{s-2} \oplus \Z_{2m_{p,s}} 
    & \text{ if }s \geq 2\text{ and }p \equiv 0,1 \bmod{3}.
\end{cases}
$$
}
\vskip.1in
\noindent

\begin{proof}
Without loss of generality, we assume $p\geq 2$:  for $p=1$, one has\footnote{The astute reader may note that this still omits the case $p=s=1$, where $\CT{p}{s}$ is a single edge; 
this case is easy.}
$\CT{1}{s} \cong \CT{2}{s-1}$,
and one can check that the group structures of $K(\Cone(\CT{1}{s}))$ and
$K(\Cone(\CT{2}{s-1}))$
predicted in the theorem are the same.

Since $p \geq 2$, the vertex $\pi_1$ in $\CT{p}{s}$ which is farthest from the star
vertices $\sigma_1,\ldots,\sigma_s$,
is a leaf, and can play the role of $v_1$ in applying
Theorem~\ref{generators}.
Hence the images $\bar{e}_{\sigma_1},\ldots,\bar{e}_{\sigma_s}$ 
generate $K(\Cone(\CT{p}{s})$.
We start by looking for relations among them,
first by looking for relations that also involve the
extra two elements 
$
\bar{e}_{\pi_{p-1}},
\bar{e}_{\pi_p}
$
and later eliminating them.
This lemma gives a sequence of relations relating $\bar{e}_{\pi_i}$ to the next $\bar{e}_{\pi_{i+1}}$.
\begin{lemma}
\label{CT-trunk-lemma}
For $i=1,2,\ldots,p-1$, one has in  $K(\Cone(\CT{p}{s}))$ that
$
F_{2i+1}\bar{e}_{\pi_i}
= 
F_{2i-1} \bar{e}_{\pi_{i+1}}. 
$
\end{lemma}
\begin{proof}[Proof of Lemma.]
Induct on $i$.  The base case $i=1$ asserts 
\begin{align*}
% F_3 \bar{e}_{\pi_1}&= F_1 \bar{e}_{\pi_2}\\
2 \bar{e}_{\pi_1}&=  \bar{e}_{\pi_2}
\end{align*}
which is the assertion that the $\pi_1^{th}$
column of $\overline{L}_{\Cone(\CT{p}{s})}$ is zero in  $K(\Cone(\CT{p}{s}))$.
In the inductive step, start with the assertion of the lemma multiplied by $-1$,
$$
-F_{2i+1}\bar{e}_{\pi_i}
= 
-F_{2i-1} \bar{e}_{\pi_{i+1}}, 
$$
and add to it the equation that asserts vanishing of the $\pi_{i+1}^{th}$ column multiplied by $F_{2i+1}$, rewritten as
$$
F_{2i+1} \bar{e}_{\pi_{i}} + F_{2i+1} \bar{e}_{\pi_{i+2}}= 3 F_{2i+1} \bar{e}_{\pi_{i+1}}.
$$
giving this equation:
\begin{align*}
F_{2i+1}\bar{e}_{\pi_{i+2}}
&= 
3 F_{2i+1} \bar{e}_{\pi_{i+1}}-F_{2i-1} \bar{e}_{\pi_{i+1}}\\
&= 
(3 F_{2i+1} -F_{2i-1}) \bar{e}_{\pi_{i+1}}\\
&= 
F_{2i+3} \bar{e}_{\pi_{i+1}}
\end{align*}
where the last equality used this easily checked Fibonacci identity 
\begin{equation}
\label{Fibonacci-identity}
3 F_{2i+1} -F_{2i-1}=F_{2i+3}.
\end{equation} 
This proves the lemma.
\end{proof}

We will now combine three types of relations: 
 the last case $i=p-1$ of Lemma~\ref{CT-trunk-lemma}, 
 isolated here as
\begin{equation}
\label{last-case-of-lemma}
F_{2p-3} \bar{e}_{\pi_{p}} 
=
F_{2p-1}\bar{e}_{\pi_{p-1}},
\end{equation}
together with the vanishing of the $\pi_p^{th}$ column of $\overline{L}_{\Cone(\CT{p}{s})}$
\begin{equation}
\label{star-center-Laplacian-column}
(s+2) \bar{e}_{\pi_p} = \bar{e}_{\pi_{p-1}} + \sum_{i=1}^s \bar{e}_{\sigma_i},
\end{equation}
and for each $i=1,2,\ldots,s$, vanishing of the $\sigma_i^{th}$ column of $\overline{L}_{\Cone(\CT{p}{s})}$
\begin{equation}
\label{double-leaves-equal-star-center}
2\bar{e}_{\sigma_i} = \bar{e}_{\pi_p} \quad \text{ for }1 \leq i \leq s.
%\notag 2(\bar{e}_{\sigma_i}-\bar{e}_{\sigma_j}) &= 0  &\text{ for }1 \leq i<j \leq s.
\end{equation}

It is helpful to express the coefficients of these relations \eqref{last-case-of-lemma}, \eqref{star-center-Laplacian-column}, \eqref{double-leaves-equal-star-center} among 
$
\bar{e}_{\pi_{p-1}},
\bar{e}_{\pi_p},
\bar{e}_{\sigma_1},
\bar{e}_{\sigma_2},
\ldots
\bar{e}_{\sigma_s},
$
in matrix form.  Consider the following $(s+2) \times (s+2)$ matrix, having rows indexed by
$
\bar{e}_{\pi_{p-1}},
\bar{e}_{\pi_p},
\bar{e}_{\sigma_1},
\bar{e}_{\sigma_2},
\ldots
\bar{e}_{\sigma_s},
$
and each of whose column vectors expresses the coefficients of
one of the relations \eqref{last-case-of-lemma}, \eqref{star-center-Laplacian-column}, \eqref{double-leaves-equal-star-center}:
$$
\bordermatrix{~ & \eqref{last-case-of-lemma} & \eqref{star-center-Laplacian-column} & \eqref{double-leaves-equal-star-center} &  \eqref{double-leaves-equal-star-center}&\cdots & \eqref{double-leaves-equal-star-center} \cr
\bar{e}_{\pi_{p-1}}& F_{2p-1} & -1& 0 & 0 & \cdots &0 \cr
\bar{e}_{\pi_{p}}  & -F_{2p-3} & s+2 &-1&-1& \cdots &-1\cr
\bar{e}_{\sigma_1} & 0         & -1 &2& 0& \cdots & 0\cr
\bar{e}_{\sigma_2} & 0         & -1 &0& 2& \cdots & 0\cr
\vdots             & \vdots    & \vdots &\vdots&  &\ddots  & \vdots\cr
\bar{e}_{\sigma_s} & 0          & -1&0& 0& \cdots & 2\cr
}
$$
Since each column vector gives a relation
valid within $K(\Cone(\CT{p}{s}))$,
by performing column operations, one can create more such valid relations.  We now use these operations to try and create a relation among only $\bar{e}_{\sigma_1},
\bar{e}_{\sigma_2},
\ldots
\bar{e}_{\sigma_s},
$
not involving the first two elements $
\bar{e}_{\pi_{p-1}},
\bar{e}_{\pi_p}$, as follows.  
Adding $F_{2p-1}$ times column $2$ to column $1$ yields this matrix:
$$
\bordermatrix{~ & ~ & ~ & ~ & ~ &~ & ~ \cr
\bar{e}_{\pi_{p-1}}& 0 & -1& 0 & 0 & \cdots &0 \cr
\bar{e}_{\pi_{p}}  & (s+2)F_{2p-1}-F_{2p-3} & s+2 &-1&-1& \cdots &-1\cr
\bar{e}_{\sigma_1} & -F_{2p-1}         & -1 &2& 0& \cdots & 0\cr
\bar{e}_{\sigma_2} & -F_{2p-1}         & -1 &0& 2& \cdots & 0\cr
\vdots             & \vdots    & \vdots &\vdots&  &\ddots  & \vdots\cr
\bar{e}_{\sigma_s} & -F_{2p-1}          & -1&0& 0& \cdots & 2\cr
}
$$
Adding $\left( (s+2)F_{2p-1}-F_{2p-3} \right)$ times column $3$ to column $1$, and also subtracting column $3$ from columns $4,5,6,\ldots,s+2$, yields the following:
$$
\bordermatrix{~ & ~ & ~ & ~ &  ~& ~ & ~ \cr
\bar{e}_{\pi_{p-1}}& \redden{0} & -1& 0 & \redden{0} & \cdots & \redden{0} \cr
\bar{e}_{\pi_{p}}  & \redden{0} & s+2 &-1& \redden{0} & \cdots & \redden{0} \cr
\bar{e}_{\sigma_1} & (2s+3)F_{2p-1}-2F_{2p-3}        & -1 &2& -2& \cdots & -2\cr
\bar{e}_{\sigma_2} & -F_{2p-1}         & -1 &0& 2& \cdots & 0\cr
\vdots             & \vdots    & \vdots &\vdots&  &\ddots  & \vdots\cr
\bar{e}_{\sigma_s} & -F_{2p-1}          & -1&0& 0& \cdots & 2\cr
}
$$
Note here the underlined zeroes, in the first two rows and in the $s$ columns indexed $1,4,5,6,\ldots,s+2$. Hence by restricting to these $s$ columns and to the last $s$ rows corresponding to $\sigma_1, \ldots, \sigma_s$, one obtains the following matrix $M$ in $\Z^{s \times s}$, whose columns represent relations
in $K(\Cone(\CT{p}{s}))$
that hold only among $
\bar{e}_{\sigma_1},
\bar{e}_{\sigma_2},
\ldots
\bar{e}_{\sigma_s}
$:

$$
M=\bordermatrix{~ & ~ &  ~& ~ &~ &  ~ \cr
\bar{e}_{\sigma_1} & (2s+3)F_{2p-1}-2F_{2p-3}        & -2& -2& \cdots & -2\cr
\bar{e}_{\sigma_2} & -F_{2p-1} & 2& 0& \cdots & 0\cr
\bar{e}_{\sigma_3} & -F_{2p-1} & 0& 2& \cdots & 0\cr
\vdots  &\vdots&  &\ddots  & \vdots\cr
\bar{e}_{\sigma_s} & -F_{2p-1} & 0&  0 &\cdots & 2\cr
}
$$
It is also useful to consider this integrally equivalent lower triangular matrix $M'$ in $\Z^{s \times s}$, obtained from $M$ by adding each of rows $2,3,\ldots,s$ to the first row:
$$
M':=\left(
\begin{matrix}
    m_{p,s} & 0 & 0& \cdots & 0 \\
    -F_{2p-1} &2 & 0 & \cdots & 0\\
    -F_{2p-1} &0 & 2 & \cdots & 0\\
    \vdots & \vdots& \vdots& \ddots& \vdots\\
    -F_{2p-1} &0 & 0 & \cdots & 2
\end{matrix}
\right).
$$
Note that since the size of the cokernel of an integral matrix is (up to a sign) its determinant, and since $M, M'$
are integrally equivalent, one has these numerical equalities
$$
|\Z^s/\im(M)|=|\det(M)|= |\det(M')|=2^{s-1} m_{p,s} = \tau(\Cone(\CT{p}{s}))=|K(\Cone(\CT{p}{s}))|.
$$
The equality of the far left and far right sides let us conclude that the surjective map
\begin{equation}
\label{s-sized-surjection}
\Z^s \twoheadrightarrow K(\Cone(\CT{p}{s}))
\end{equation}
sending the $i^{th}$ standard basis vector of $\Z^s$ to $\bar{e}_{\sigma_i}$ for $i=1,2,\ldots,s$,
will descend to an isomorphism 
$$
\Z^s/\im(M) \cong K(\Cone(\CT{p}{s})),
$$
because it is surjective, and both sides
have the same cardinality.  Thus $K(\Cone(\CT{p}{s}))$ is the same as the integer cokernel of $M$ or $M'$.  Our last step is to analyze the cokernel of $M'$, in the two cases of the theorem.

\vskip.1in
\noindent
{\sf Case 1.} $s =1 $, or $p \equiv 2 \bmod{3}$ so that $F_{2p-1}$ is even.

When $s=1$, the matrix $M'=[ m_{p,s} ]$ is $1 \times 1$, and the assertion is easily checked.  
When $F_{2p-1}$ is even, one can add $\frac{F_{2p-1}}{2}$ times each of columns $2,3,\ldots,s-1$ of $M'$ to the first column,  giving this integrally equivalent diagonal matrix, with cokernel $\Z_2^{s-1} \oplus \Z_{m_{p,s}}$, as desired:
$$
\left(
\begin{matrix}
    m_{p,s} & 0 & 0& \cdots & 0 \\
    0 &2 & 0 & \cdots & 0\\
    0 &0 & 2 & \cdots & 0\\
    \vdots & \vdots& \vdots& \ddots& \vdots\\
    0 &0 & 0 & \cdots & 2
\end{matrix}
\right).
$$

\vskip.1in
\noindent
{\sf Case 2.} $s\geq 2$ and $p \equiv 0,1 \bmod{3}$ so that $F_{2p-1}$ is odd.

Here adding $\frac{F_{2p-1}-1}{2}$
times each column $2,3,\ldots,s$ of $M'$ to the first column  results in this $s \times s$ matrix:
$$
\left(
\begin{matrix}
    m_{p,s} & 0 & 0& \cdots & 0 \\
    -1 &2 & 0 & \cdots & 0\\
    -1 &0 & 2 & \cdots & 0\\
    \vdots & \vdots& \vdots& \ddots& \vdots\\
    -1 &0 & 0 & \cdots & 2
\end{matrix}
\right).
$$
One can then use the $(2,1)$ entry $-1$   as a pivot to eliminate all of the entries in the first column.  The result has the same cokernel as its $(s-1)\times (s-1)$ minor deleting row $2$ and column $1$:
$$
\left(
\begin{matrix}
    2m_{p,s} & 0 & 0& \cdots & 0 \\
    -2 &2 & 0 & \cdots & 0\\
    -2 &0 & 2 & \cdots & 0\\
    \vdots & \vdots& \vdots& \ddots& \vdots\\
    -2 &0 & 0 & \cdots & 2
\end{matrix}
\right).
$$
Similar to Case 1, this last matrix has cokernel
$\Z_2^{s-2} \oplus \Z_{2m_{p,s}}$, as desired.
\end{proof}

\begin{remark} \rm
\label{coconut-tree-mu-remark}
As a test case for
Question~\ref{mu-bound-question},
we compare
the exact values of $\mu(\Cone(T))$ for coconut
trees $T=\CT{p}{s}$
implied by Theorem~\ref{coconut-theorem}(ii)
with the upper bound of $\ell(T)-1$ from Theorem~\ref{generators}.

In doing this comparison, one can assume without loss
of generality that $p \geq 2$, due to the
isomorphism $\CT{1}{s} \cong \CT{2}{s-1}$
pointed out in the proof of Theorem~\ref{coconut-theorem}(ii).
For $p \geq 2$, one always has $\ell(\CT{p}{s})=s+1$.
When computing $\mu(\Cone(\CT{p}{s}))$,
one must remember that
$
\Z_2 \oplus \Z_{m_{p,s}} \cong \Z_{2m_{p,s}}
$
whenever $m_{p,s}$ is odd; the latter holds
if and only if both $s$ is odd and $F_{2p-1}$ is odd, or equivalently, if both $s$ is odd and $p \equiv 0,1 \bmod{3}$.  Thus
Theorem~\ref{coconut-theorem}(ii) implies for
coconut trees $\CT{p}{s}$ with 
$p \geq 2$ that
$$
\mu(K(\Cone(\CT{p}{s})))=
\begin{cases}
\ell(\CT{p}{s})-1=s & \text{ if }s=1
 \text{ or if }p \equiv 2 \bmod{3} \\
 & \text{ or if }s\geq 2\text{ is even and }p \equiv 0,1 \bmod{3},\\
 & \\
\ell(\CT{p}{s})-2=s-1& \text{ if }s \geq 3\text{ is odd and }p \equiv 0,1 \bmod{3}.
\end{cases}
$$
The distinction between the two cases here seems already a bit subtle.
\end{remark}

%%%%%%%%%%%%%%%%%%%%%%%%
\section{Proof of Theorem~\ref{tutte-bounds-theorem}}
\label{tutte-bounds-section}

This section proves Theorem ~\ref{tutte-bounds-theorem},
giving coefficientwise lower and upper bounds for the Tutte polynomial evaluations $T_G(1,y)$ that enumerate recurrent chip configurations by weight, when $G$ is the cone over
a tree having $n$ vertices. 

Recall the statement,
with the two relevant coefficientwise inequalities to be proven marked (A),(B).
\vskip.1in
\noindent
{\bf Theorem~\ref{tutte-bounds-theorem}.}
{\it 
    For a tree $T$ on $n \geq 2$ vertices, 
    the weight enumerator of recurrent chip configurations
    $T_{\Cone(T)}(1,y)$ from \eqref{Merino-theorm} satisfies
the following coefficientwise inequalities as
polynomials in $\Z[y]$:
\begin{equation}
   T_{\Cone(S_n)}(1,y) \,\,  
   \overset{(A)}{\leq} 
   \,\, T_{\Cone(T)}(1,y) \,\, 
   \overset{(B)}{\leq} \,\, T_{\Cone(P_n)}(1,y).
\end{equation}
In particular, setting $y=1$, this implies
$$
\begin{array}{cccccl}
\left( 2^{n-2}(n+1) =\right) \,\, \tau(\Cone(S_n)) &\leq& \tau (\Cone(T))&\leq &\tau(\Cone(P_n)) \,\, \left(= F_{2n}\right)\\
%\Vert& & & &\Vert\\ 2^{n-2}(n+1)& & & &F_{2n}
% & & & & (\approx \frac{1}{\sqrt{5}} \cdot (\varphi^2)^n)&
\end{array}
$$
}
\vskip.1in
\noindent

The next two subsections deal with inequality (A) and (B) in turn.

%%%%%%
\subsection{Proof of inequality (A)}
This inequality will follow from a stronger
assertion, Lemma~\ref{star-inequality-lemma}
below, that also proves the following assertion from the Introduction:  one has
$
T_{\Cone(S_n)}(1,y)=S_n(y),
$
where $S_n(y)$ is a polynomial with $S_n(1)=2^{n-2}(n+1)=\tau(\Cone(S_n))$ defined by this
recursion:
$$
S_1(y):=1 \quad \text{ and }\quad 
S_n(y):=(y+1) S_{n-1}(y) + 2^{n-2} \text{ for }n \geq 2.
$$

\begin{lemma}
\label{star-inequality-lemma}
For any leaf vertex $\ell$ in any tree $T$ with $n$ vertices, one has a coefficientwise inequality in $\Z[y]$
$$
T_{\Cone(T)}(1,y) \,\,
\geq \,\,
(y+1) \cdot T_{\Cone(T-\ell)}(1,y) + 2^{n-2},
$$
with equality if $T=S_n.$
\end{lemma}

Lemma~\ref{star-inequality-lemma} plays two roles.
On one hand, the equality case in Lemma~\ref{star-inequality-lemma} immediately proves the assertion $T_{\Cone(S_n)}(1,y)=S_n(y)$
by induction on $n$.  On the other hand, 
the inequality in Lemma~\ref{star-inequality-lemma} also 
proves inequality (A) by induction on $n$,
with this inductive step:
$$
\begin{aligned}
T_{\Cone(T)}(1,y)
&\geq (y+1) T_{\Cone(T-\ell)}(1,y)+2^{n-2} \\
&\geq (y+1) T_{\Cone(S_{n-1})}(1,y)+2^{n-2} \quad \text{ (since }T_{\Cone(T-\ell)}(1,y)\geq T_{\Cone(S_{n-1})}(1,y)\text{ by induction)}\\
&=(y+1) S_{n-1}(y) +2^{n-2}
=S_n(y)
=T_{\Cone(S_n)}(1,y).
\end{aligned}
$$

\begin{proof}[Proof of Lemma~\ref{star-inequality-lemma}.]
Recall \eqref{ext-act-formula} allows us to choose
a linear ordering $\prec$ on the edges of $\Cone(T)$
and express
\begin{equation}
\label{sum-to-tripartition}
 T_{\Cone(T)}(1,y) = \sum_S y^{\extact_\prec(S)}.
\end{equation}
where $S$ runs through all spanning trees of $\Cone(T)$.
We choose a particular linear ordering $\prec$ as follows.
Let $v_0$ be the cone vertex, not in $T$, and let $u_0$ be the unique vertex in $T$ attached to the leaf vertex $\ell$.  Choose a linear order $\prec$ with these properties:
\begin{itemize}
    \item All edges $\{v_0,u\}$ incident to the cone vertex $v_0$ are $\prec$-smaller than all edges $\{u,u'\}$ inside $T$.
     \item In particular, $\{v_0,\ell\}$ has $\prec$-smallest label (call it $-\infty$), and
    $\{\ell,u_0\}$ the $\prec$-largest (call it $+\infty$). 
\end{itemize}
The figure below shows an example of $T, \ell, u_0$ and
$\Cone(T)$:

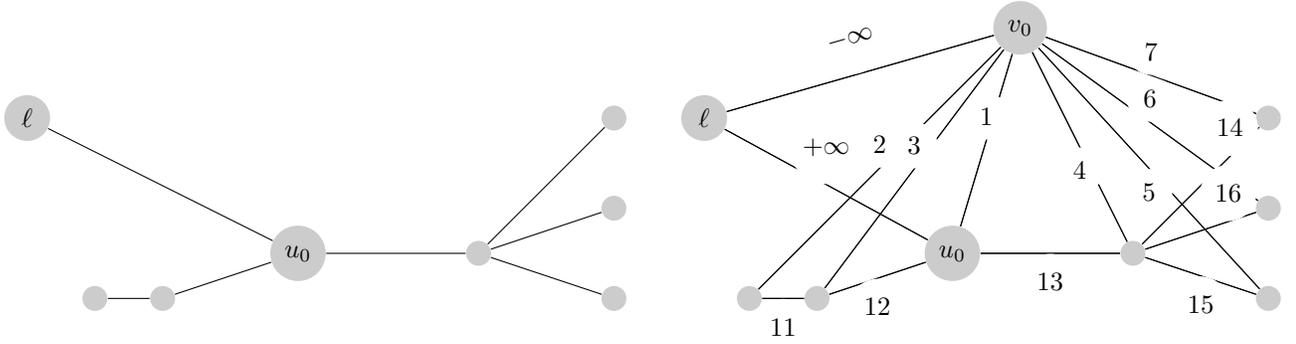
\begin{figure}
\begin{center}
\begin{minipage}{.2\textwidth}
 \hspace*{-1.6\linewidth}
\begin{tikzpicture}
  [scale=.30,auto=left,every node/.style={circle,fill=black!20}]
%  \node (l) at (-6,16) {$\ell$};
  \node (l) at (-19,16) {$\ell$};
  \node (u0) at (-7,10) {$u_0$};
  \node (u1) at (-16,8) {};
  \node (u2) at (-13,8) {};
  \node (u6) at (7,16) {};
  \node (u5) at (7,12) {}; 
  \node (u4) at (7,8) {}; 
  \node (u3) at (1,10) {};
\foreach \from/\to in {l/u0,u1/u2,u2/u0,u0/u3,u3/u6,u3/u4,u3/u5}
    \draw (\from) -- (\to);
  \node (v0) at (25,20) {$v_0$};
%  \node (l) at (19,16) {$\ell$};
  \node (l) at (11,16) {$\ell$};
  \node (u0) at (22,10) {$u_0$};
  \node (u1) at (13,8) {};
  \node (u2) at (16,8) {};
  \node (u6) at (36,16) {};
  \node (u5) at (36,12) {}; 
  \node (u4) at (36,8) {}; 
  \node (u3) at (30,10) {};
\foreach \from/\to in {l/u0,u1/u2,u2/u0,u0/u3,u3/u6,u3/u4,u3/u5,
v0/l,v0/u1,v0/u2,v0/u3,v0/u4,v0/u5,v0/u6,v0/u0}
    \draw (\from) -- (\to);
\draw (v0) -- (u0) node [midway, above,fill=white] {$1$};
\draw (v0) -- (u1) node [midway, above,fill=white] {$2$};
\draw (v0) -- (u2) node [midway, above,fill=white] {$3$};
\draw (v0) -- (u3) node [midway, below,fill=white] {$4$};
\draw (v0) -- (u4) node [midway, below,fill=white] {$5$};
\draw (v0) -- (u5) node [midway, above,fill=white] {$6$};
\draw (v0) -- (u6) node [midway, above,fill=white] {$7$};

\draw (v0) -- (l) node [midway, sloped, above,fill=white] {$-\infty$};
\draw (l) -- (u0) node [midway, below, above, fill=white] {$+\infty$};

\draw (u1) -- (u2) node [midway, below,fill=white] {$11$};
\draw (u2) -- (u0) node [midway, below,fill=white] {$12$};
\draw (u3) -- (u0) node [midway, below,fill=white] {$13$};
\draw (u3) -- (u6) node [near end, above,fill=white] {$14$};
\draw (u3) -- (u4) node [midway, below,fill=white] {$15$};
\draw (u3) -- (u5) node [near end, above,fill=white] {$16$};

\end{tikzpicture}
\end{minipage}
\end{center}
\caption{The left depicts a tree $T$ with a leaf $\ell$ whose unique neighbor is labeled $u_0$.  The right depicts $\Cone(T)$, along with one choice of a labeling of its edges via the linear ordered alphabet $-\infty <1<2<3<4<5<6<7<11<12<13<14<15<16<+\infty$, as in the proof.}
\end{figure}

Since vertex $\ell$ has exactly two neighbors $v_0, u_0$ in $\Cone(T)$, one can break the sum \eqref{sum-to-tripartition} into
three subsums, according to whether the spanning
tree $S$ uses only edge $\{\ell,v_0\}$, or 
only edge $\{\ell,u_0\}$, or both edges:
\begin{equation}
 \label{tripartitioned-sums}   
 T_{\Cone(T)}(1,y) 
 = \sum_{\substack{S:\\
           \{\ell,v_0\} \in S,\\
           \{\ell,u_0\} \not\in S}}
           y^{\extact_\prec(t)}
 + \sum_{\substack{S:\\
           \{\ell,v_0\} \not\in S,\\
           \{\ell,u_0\} \in S}} y^{\extact_\prec(S)}
+ \sum_{\substack{S:\\
           \{\ell,v_0\} \in S,\\
           \{\ell,u_0\} \in S}} y^{\extact_\prec(S)}
\end{equation}
For trees $S$ in the first sum or in the second sum
of \eqref{tripartitioned-sums}, $\ell$ is a leaf of $S$, and the map $S \mapsto S - \ell$ bijects them with spanning trees $S-\ell$ in $\Cone(T-\ell)$.  As explained below, our choice of $\prec$ makes
$$
\extact_\prec(S)=
\begin{cases}
    \extact_\prec(S-\ell) & \text{ for }S\text{ in the first sum,}\\
    \extact_\prec(S-\ell)+1 & \text{ for }S\text{ in the second sum.}\\
\end{cases}
$$
We claim that for $S$ in the first sum, $\ExtAct_\prec(S)=\ExtAct(S -\ell)$,
while for $S$ in the second sum, $\ExtAct_\prec(S)=\ExtAct(S -\ell) \sqcup \{\{\ell,v_0\}\}$.  In both cases, this holds because
the two extra edges $\{ \ell,v_0\}, \{\ell,u_0\}$ of $\Cone(T)$ not present in $\Cone(T -\ell)$ carry 
the labels $-\infty, +\infty$.

Consequently, the first two sums in \eqref{tripartitioned-sums} together sum to $(1+y)T_{\Cone(T-\ell)}(1,y)$.
Therefore, the desired inequality asserted by Lemma~\ref{star-inequality-lemma} will follow if one can exhibit, within the third sum of \eqref{tripartitioned-sums}, at
least $2^{n-2}$ summands of the form $y^0$, that is,  where $S$ has  $\extact_\prec(S)=0$.  Here is a recipe to produce $2^{n-2}$ such spanning trees $S$, starting with any subset $F$ of the $n-2$ edges of the tree $T  - \ell$.  Since $F$ is a subset of edges in a tree, it forms a forest, whose connected components partition the vertices of $T  - \ell$.  Label this vertex partition 
$$
V(T  - \ell) =U_0 \sqcup U_1 \sqcup \cdots \sqcup U_r 
$$
so that $u_0$ lies in $U_0$.  In each of the other components $U_i$ for $i=1,2,\ldots,r$, pick the unique vertex $u_i$ in $U_i$ so that the edge $\{v_0,u_i\}$ is $\prec$-smallest among all edges $\{v_0,u\}_{u \in U_i}$.  We check that the set 
$$
S:=F \cup 
\{\{v_0,\ell\},
\{u_0,\ell\}\}\cup \{ \{v_0,u_i\} \}_{i=1,2,\ldots,r}
$$
is a spanning tree with $\extact_\prec(S)=0$ contributing $y^0$ the third sum.  To see this, we check that any edge $f$ of $\Cone(T)$ outside of $S$ cannot lie in $\ExtAct(S)$. 
\vskip .1in
\noindent
{\sf Case 1.} $f$ lies in $T$.

Then $f$ connects two components $U_i, U_j$ of the forest $F$ for $i \neq j$. But then $f \not\in \ExtAct(S)$ since either of $e=\{u_i,v_0\}$ or $e=\{u_j,v_0\}$ has $S - e \cup f$ a spanning tree, and $e \prec f$. 

\vskip .1in
\noindent
{\sf Case 2.}
$f$ does not lie in $T$, say $f=\{v_0,u\}$ with
$u$ lying component $U_i$ of the forest $F$.
\vskip .1in
\noindent
{\sf Case 2a.} $i \geq 1$, so $U_i \neq U_0$.
Then $f \not\in \ExtAct(S)$ since $e=\{u_i,v_0\}$ has $S - e \cup f$ a spanning tree, and $e \prec f$.
\vskip .1in
\noindent
{\sf Case 2b.} $i=0$ so $u \in U_0$.
Then $f \not\in \ExtAct(S)$ since  $e=\{\ell,v_0\}$  has $S - e \cup f$ a spanning tree, and $e \prec f$.
\vskip .1in

  Lastly, when $T=S_n$, we claim equality holds in the Lemma because the third sum of \eqref{tripartitioned-sums} contains {\it only} the $2^{n-2}$ different trees $S$ constructed above. Spanning trees in the third sum contain both edges $\{\ell,v_0\}, \{\ell,u_0\}$ incident to leaf $\ell$, so they 
  must omit the edge $\{v_0,u_0\}$, and for each of the $n-2$ other leaves $u_1,\ldots,u_{n-2}$ of $S_n$, they must contains exactly one of the two edges $\{u_i,v_0\}, \{u_i,u_0\}$, but not both.
\end{proof}

%%%%%%%
\subsection{Proof of inequality (B)}

Recall Corollary~\ref{fan-doubled-fan-corollary} already showed that the recursively defined
polynomial
$$
F_0(y):=0, F_1(y):=1, \quad
F_n(y) := 
\begin{cases}
F_{n-1}(y) + F_{n-2}(y) &\text{ for }n\text{ even,}\\
F_{n-1}(y) + yF_{n-2}(y) &\text{ for }n\text{ odd.}
\end{cases}
$$
has the following interpretations:
\begin{align*}
T_{\Cone(P_n)}(1,y)&=F_{2n}(y),\\
T_{\Coneplus{\ell}(P_n)}(1,y)&=F_{2n+1}(y).
\end{align*}

Consequently, inequality (B) will follow from this stronger lemma, that facilitates an
inductive proof.

\begin{lemma}
\label{intertwining-induction-lemma}
Let $v$ be any vertex in a tree $T$ on $n$ vertices,
and let $v_0$ be the cone vertex of $\Cone(T)$. 
Then one has these coefficientwise inequalities:
\begin{itemize}
    \item[(i)] 
    $
    T_{\Cone(T)}(1,y) \leq F_{2n}(y),
    $
    with equality if $T=P_n$.
    \item[(ii)]
    $
    T_{\Cone(T)/\{v_0,v\}}(1,y) \leq F_{2n-1}(y),
    $
    with equality if $T=P_n$ and $v$ is an end vertex of $P_n$.
\end{itemize}
\end{lemma}

\begin{proof}[Proof.]
The cases of equality in (i),(ii) are both implied by
Corollary~\ref{fan-doubled-fan-corollary}, since
one can check that when $v$ is an end vertex of $P_n$, then 
$\Cone(T)/\{v_0,v\} \cong \Coneplus{\ell}(P_{n-1})$ where
$\ell$ is an end vertex of $P_{n-1}$.

We prove the inequalities in (i),(ii) 
simultaneously via induction on $n$. 
Both base cases $n=1$ are easy.

\vskip.1in
\noindent
{\sf Inductive step for proof of (i) with $n \geq 2$.}
Choose $v$ to be any {\it leaf} vertex of $T$, with unique
neighbor vertex $w$ in $T$.  One can use \eqref{Tutte-recursive-definition} twice:  first one can perform deletion and contraction
on the edge $\{v_0,v\}$, and then in 
the graph $\Cone(T)/\{v_0,v\}$
(which contains a single contracted vertex labeled $v_0v$) one 
can perform deletion and
contraction on the edge $\{v_0v, w\}$.  This gives the following
calculation, explained further below:
\begin{align}
T_{\Cone(T)}(x,y)
\notag
&=T_{\Cone(T)\setminus \{v_0,v\}}(x,y) 
   + T_{\Cone(T)/\{v_0,v\}}(x,y)\\
\notag
&=T_{\Cone(T)\setminus \{v_0,v\}}(x,y) 
   + T_{\left(\Cone(T)/\{v_0,v\}\right) \setminus \{v_0v,w\}}(x,y)
   + T_{\left(\Cone(T)/\{v_0,v\}\right)/ \{v_0v,w\} }(x,y)\\
\label{two-deletion-contractions-round1}
&=x\cdot T_{\Cone(T-v)}(x,y) 
   + T_{\Cone(T-v)}(x,y)
   + y\cdot T_{\Cone(T-v)/\{v_0,w\} }(x,y)
\end{align}
The last equality used three graph isomorphisms: 
\begin{itemize}
 \item $\Cone(T)\setminus \{v_0,v\}$ is isomorphic to $\Cone(T-v)$
with an isthmus edge $\{v,w\}$ attached to $w$,
    \item $\left(\Cone(T)/\{v_0,v\} \right) \setminus \{v_0v,w\}$ is isomorphic to $\Cone(T-v)$, 
\item $\left( \Cone(T)/\{v_0,v\} \right)/ \{v_0v,w\}$ is isomorphic to
$\Cone(T-v)/\{v_0,w\}$ with a loop attached to vertex $v_0$.
\end{itemize}
Therefore, setting $x=1$ in \eqref{two-deletion-contractions-round1}, and then applying (i),(ii) inductively to vertex $w$ in $T-v$, one obtains
\begin{align*}
T_{\Cone(T)}(1,y)
&= T_{\Cone(T-v)}(1,y) 
   + T_{\Cone(T-v)}(1,y)
   + y\cdot T_{\Cone(T-v)/\{v_0,w\} }(1,y)\\
&\leq F_{2(n-1)}(y) + \left( F_{2(n-1)}(y) + y\cdot F_{2(n-1)-1}(y)\right)\\
&= F_{2n-2)}(y) +  F_{2n-1}(y)\\
&= F_{2n}(y)
\end{align*}
where the last two equalities used the 
two cases of the recursive definition of $F_n(y)$.

\vskip.1in
\noindent
{\sf Inductive step for proof of (ii) with $n \geq 2$.}
Given $v$, choose any neighbor vertex $w$ of $v$ in $T$.  Deleting the edge $\{v,w\}$ from $T$ leaves two connected components, a subtree $T_v$ containing $v$ and a subtree $T_w$ containing $w$, having $n_v, n_w$ vertices, respectively, with $n_v+n_w=n$.

Now the graph $\Cone(T)/\{v_0,v\}$ contains a single contracted vertex labeled $v_0v$, and contains two parallel copies
$e,e'$ of the edge $\{v_0v,w\}$.  Performing deletion and contraction on one copy $e$ of this parallel edge, and 
using \eqref{Tutte-recursive-definition}, gives the following
calculation, explained further below:
\begin{align}
\notag 
T_{\Cone(T)/\{v_0,v\}}(x,y)
&=T_{\left( \Cone(T)/\{v_0,v\}\right) \setminus e}(x,y)
+ T_{\left( \Cone(T)/\{v_0,v\} \right) / e }(x,y)\\
\label{two-deletion-contractions-round2}
&=T_{\Cone(T_v)/\{v_0,v\}}(x,y) \cdot T_{\Cone(T_w)}(x,y)
+ y\cdot T_{\Cone(T/\{v,w\})/\{v_0,vw\}}(x,y).
\end{align}
The last equality used two facts.
The first fact is that one can apply
\eqref{cut-vertex-Tutte-fact} to $\left( \Cone(T)/\{v_0,v\} \right) \setminus e$, since it has the contracted vertex $v_0v$ as a cut-vertex.  Removing vertex $v_0v$ creates two components $C_v,C_w$, where 
 $C_v$ together with $v_0v$ forms a graph isomorphic to $\Cone(T_v)/\{v_0,v\}$, and $C_w$ together with 
  $v_0v$ forms a graph isomorphic to $\Cone(T_w)$.
 The second fact is that
  $\left(\Cone(T)/\{v_0,v\} \right)/e$ is isomorphic to $\Cone(T/\{v,w\})/\{v_0,vw\}$ with a loop $e'$ attached to $v_0vw$.
  
  Thus setting $x=1$ in \eqref{two-deletion-contractions-round2} and then using (i),(ii) inductively
gives
\begin{align*}
T_{\Cone(T)/\{v_0,v\}}(1,y)
&=T_{\Cone(T_v)/\{v_0,v\}}(1,y) \cdot
  T_{\Cone(T_w)}(1,y)
  + y \cdot T_{\Cone(T/\{v,w\})/\{v_0,vw\}}(1,y)\\
  &\leq F_{2n_v-1}(y) \cdot F_{2n_w}(y) + y  F_{2(n-1)-1}(y) \quad\text{ by induction}\\
  &\leq F_{2(n_v+n_w)-2}(y) + y F_{2n-3}(y) \quad \quad\text{ by Lemma~\ref{Fibonacci_ineq} below}\\
  &= F_{2n-2}(y) + y F_{2n-3}(y)
  \quad\quad\quad\quad\quad\text{ since }n_v+n_w=n\\
  &= F_{2n-1}(y).
\end{align*}
where the last equality used the recursive definition of $F_n(y)$.
\end{proof}

\begin{lemma}\label{Fibonacci_ineq}
For $n,m \geq 1$ not both even, one has a coefficientwise inequality
$
F_n(y)F_m(y) \leq F_{n+m-1}(y).
$
\end{lemma}
\begin{proof}
Induct on $n+m$, in two cases. If $n,m$ are both odd then 
$$
\begin{aligned}
F_n(y)F_m(y) = F_n(y)[F_{m-1}(y)+yF_{m-2}(y)]
&= F_n(y)F_{m-1}(y)+yF_n(y)F_{m-2}(y)\\
&\leq F_{n+m-2}(y) + yF_{n+m-3}(y) \quad  \text{ by induction}\\
&= F_{n+m-1}(y).
\end{aligned}
$$
If $n$ is even and $m$ is odd then 
$$
\begin{aligned}
    F_n(y)F_m(y) 
    = F_m(y)[F_{n-1}(y)+F_{n-2}(y)]
    &= F_m(y)F_{n-1}(y)+F_m(y)F_{n-2}(y)\\
&\leq F_{n+m-2}(y) + F_{n+m-3}(y) \quad\text{ by induction}\\
&= F_{n+m-1}(y).\qedhere
\end{aligned}
$$
\end{proof}

%\begin{remark} \rm    Lemma~\ref{Fibonacci_ineq} at $y=1$ becomes    $F_n F_m \leq F_{n+m-1}$,    and follows from a Fibonacci identity   $F_n F_m + F_{m-1} F_{n-1}= F_{n+m-1}$.    We have not tried to generalize this identity    to incorporate the variable $y$. \end{remark}

%%%%%%%%%%%%%%%%%%
\begin{section} {Discussion, remarks, questions}
\label{remarks-section}

As mentioned in the Introduction, our goal was to use the graphs $G=\Cone(T)$ as test cases for the Main Question on how the ``shape" 
of a graph $G$ affects the sandpile group $K(G)$, including its size $\tau(G)=|K(G)|$, its minimal number of generators $\mu(G)$, and the coefficients of $T_G(1,y)$.  We discuss this further here.

\subsection{Csikv\'ari's poset on trees}
Theorems~\ref{generators}, \ref{generators-bounds-theorem} suggest that the nunber of leaves $\ell(T)$ in the tree $T$ has some {\it positive correlation} with $\mu(\Cone(T))$, while Theorem~\ref{tutte-bounds-theorem} suggests $\ell(T)$ has some {\it negative correlation} with both $\tau(\Cone(T))$ and the coefficients of $T_{\Cone(T)}(1,y)$.

Seeing this, it is natural to consider work of 
P. Csikv\'ari \cite{OnaPosetofTrees, Csikvari2} which introduced an interesting ranked poset structure on the set of all trees with $n$ vertices, having the path $P_n$ as its
unique minimal element, the star $S_n$ as its unique maximal, with rank function $\mathrm{rank}(T)=\ell(T)-2$; see \cite[Thm. 2.4, Cor. 2.5]{OnaPosetofTrees}. 
Many interesting graph invariants for trees behave in a monotone
fashion with respect to this poset, so that they
are bounded by the values on $P_n, S_n$; see \cite[\S 11]{Csikvari2}.
His poset is the transitive closure of
covering relations $T \lessdot T'$ coming from
this operation called a {\it generalized tree shift} \cite[Defn. 2.1]{OnaPosetofTrees}:
\begin{quote}
 Let $(x,y)$ be an ordered pair of distinct vertices in a tree $T$ such that
 all interior vertices on the path from $x$ to $y$
have degree two in $T$. Let $x'$ be the unique neighbor of $y$ on this path to $x$; possibly $x'=x$.  Then the {\it generalized tree shift} $T'$ (of $T$ at $x,y$) is obtained from $T$ by deleting all edges between $y$ and its $T$-neighbors, except for the edge $\{y,x'\}$, and adding in edges between $x$ and all of the former $T$-neighbors of $y$ except $x'$.
\end{quote}

For example, one can check that for $p \geq 3$, one has $CT(p,s) \lessdot CT(p-1,s+1)$ in Csikv\'ari's poset on trees with $n=p+s$ vertices.
Here is his poset on trees with $n=7$ vertices:

\begin{center}
\includegraphics[scale=0.15]{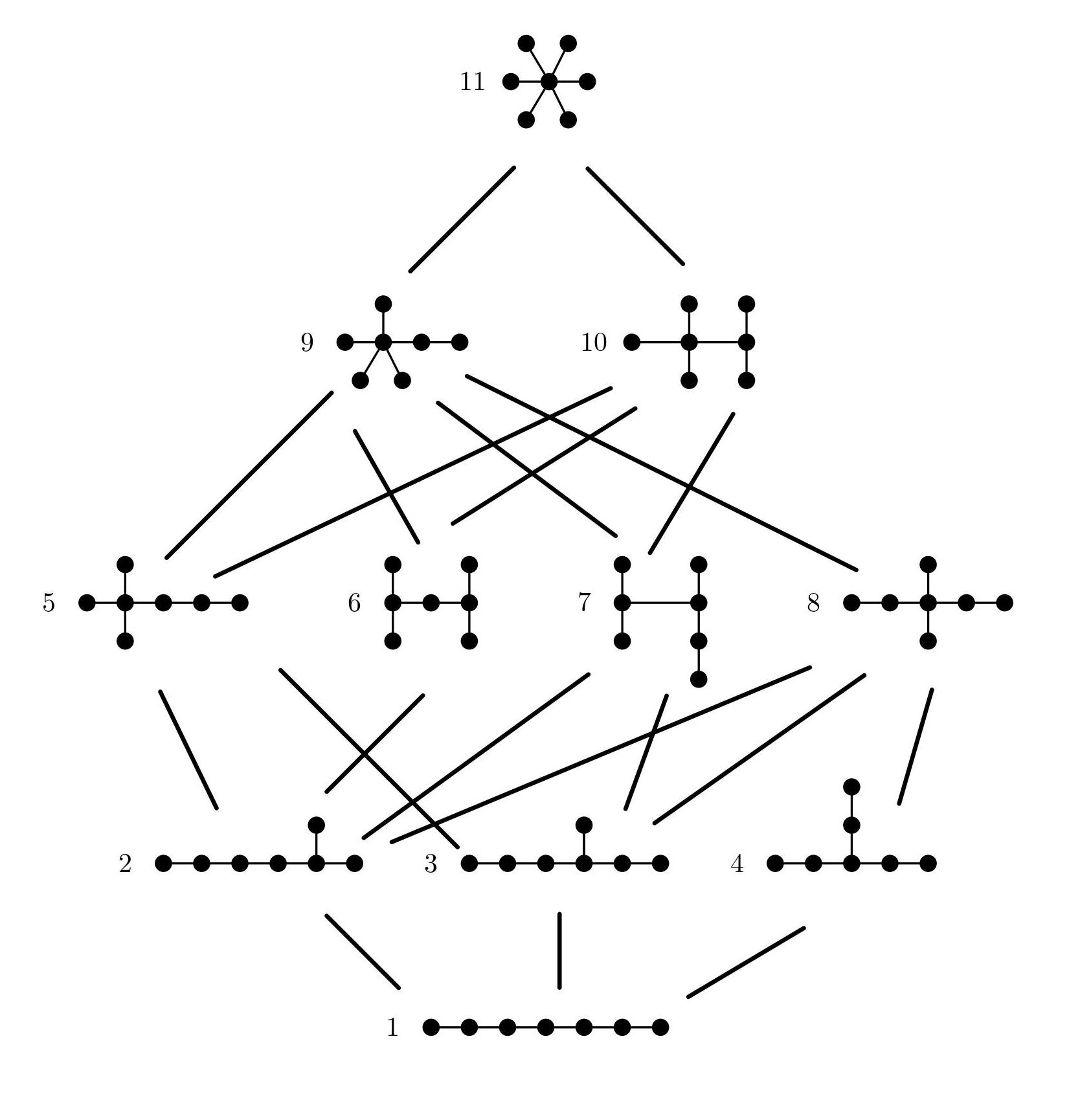}

\end{center}
%\caption{The poset of all trees on 7 vertices using Csikv\'ari's ordering}
%\label{fig:csikvari}

Here is the data 
on $K(\Cone(T))$ and $T_{\Cone(T)}(1,y)$
for the trees $T$ on $7$ vertices shown above.

\begin{center}
\begin{tabular}{ |c|c|c| } 
 \hline
%        &    &   \\
 $T$& $K(\Cone(T))$ & $T(1,y)$ \\
  \hline\hline
 1 & $\Z_{377}$ & $y^6+7y^5+26y^4+63y^3+104y^2+122y+64$ \\
  \hline
  2 & $\Z_2\oplus \Z_{178}$ & $y^6+7y^5+25y^4+59y^3+96y^2+104y+64$ \\
  \hline
 3 & $\Z_{353}$ & $y^6+7y^5+25y^4+58y^3+94y^2+104y+64$ \\
  \hline
 4 & $\Z_{5}\oplus \Z_{70}$ & $y^6+7y^5+25y^4+57y^3+92y^2+104y+64$ \\
  \hline
 5 & $\Z_2\oplus \Z_{162}$ & $y^6+7y^5+24y^4+53y^3+83y^2+92y+64$ \\
  \hline
 6 & $\Z_{4}\oplus \Z_{84}$ & $y^6+7y^5+24y^4+55y^3+89y^2+96y+64$ \\
  \hline
 7 & $\Z_{332}$ & $y^6+7y^5+24y^4+54y^3+86y^2+96y+64$ \\
  \hline
 8 & $\Z_{320}$ & $y^6+7y^5+24y^4+52y^3+80y^2+92y+64$ \\
  \hline
 9 & $(\Z_2)^2\oplus \Z_{72}$ & $y^6+7y^5+23y^4+47y^3+68y^2+78y+64$ \\
  \hline
 10 & $(\Z_2)^2\oplus \Z_{76}$  & $y^6+7y^5+23y^4+49y^3+76y^2+84y+64$ \\
  \hline
 11 & $(\Z_2)^4\oplus \Z_{16}$  & $y^6+7y^5+22y^4+42y^3+57y^2+63y+64$ \\ 
 \hline
\end{tabular}
\end{center}
%\caption{The poset of all trees on 7 vertices using Csikv\'ari's ordering}

The data suggests that as one goes up in Csikv\'{a}ri's order, $\mu(\Cone(T))$ often increases,
but this is not always the case.
For example $T_4 \lessdot T_8$, but $\mu(\Cone(T_4))=2 > 1=\mu(\Cone(T_8))$.

On the other hand, computations using {\tt Sage/Cocalc} code written by Trevor Karn on trees up to $13$ vertices suggested the next conjecture, which appeared in the first {\tt arXiv} version of this paper, and was subsequently proven by Changxin Ding \cite{Ding}.

\begin{conjecture} 
If $T<T'$ in  Csikv\'{a}ri's poset, then
$T_{\Cone(T)}(1,y) \geq T_{\Cone(T')}(1,y)$ coefficientwise in $\Z[y]$.
\end{conjecture}

\subsection{Degree sequences and line graph edge density}
As one goes upward from the path $P_n$ to the star $S_n$ in Csikv\'ari's order on trees $T$
with $n$ vertices, certain measures of the ``shape" of the graph $G=\Cone(T)$ move upward.
For example, the {\it line graph} 
$$
\Line(G)=
( V_{\Line(G)}, E_{\Line(G)})
$$
for $G=\Cone(T)$
has its edge density $\frac{|E_{\Line(G)}|}{|V_{\Line(G)}|}$ increasing as one goes up.
Also, the weakly decreasing degree sequence 
$d(G)=(d_1 \geq d_2 \geq \cdots \geq d_{n+1})$ moves upward in the {\it dominance} or {\it majorization} order on
partitions.

\begin{question}
For {\bf any} graph $G=(V,E)$, not necessarily a cone on a tree, consider these two parameters:
\begin{itemize}
 \item {\it line graph edge density} $\frac{|E_{\Line(G)}|}{|V_{\Line(G)}|}$, or
\item height of the degree sequence 
$d(G)=(d_1 \geq d_2 \geq \cdots)$ in the dominance order on partitions of $2|E|$.
   \end{itemize}
Do either of these parameters of the graph $G$ exhibit
\begin{itemize}
    \item 
positive correlation with $\mu(G)$, or
\item 
negative correlation with $\tau(G)$
and with the coefficients of $T_G(1,y)$?
\end{itemize}
\end{question}

%%%%%%%%%%%%%%%%%%%%%%%%%%%%%%%%%%%%%%%%%%%%%
\section*{Acknowledgements}
The authors thank J. Urschel for showing them a proof of the numerical inequalities
\eqref{Urschel-result}, Angel Chavez and Trevor Karn for their tireless help with {\tt Sage/Cocalc}, and Changxin Ding and two anonymous referees for helpful edits.
\end{section}
%%%%%%%%%%%%%%%%%%%%%%%%%%%%%%%%%%%%%%%%%%%%%
\section*{Conflicts of Interest}
All authors declare that they have no conflicts of interest.

%%%%%%%%%%%%%%%%%%%%%%%%%%%%%%%%%%%%%%%%%%%%%
\section*{Data Availability Statement}
Data availability is not applicable to this article as no new data were created or analyzed in this study.

\bibliographystyle{alpha}
\bibliography{ref}
    
\end{document}